\documentclass{birkmult}
 \newtheorem{theorem}{Theorem}[section]
 \newtheorem{corollary}[theorem]{Corollary}

 \theoremstyle{definition}
 \newtheorem{property}[theorem]{Property}
 \newtheorem{definition}[theorem]{Definition}
 \theoremstyle{remark}

 \newtheorem*{exercise}{Exercise}
 \numberwithin{equation}{section}
 \usepackage{fancybox}

\begin{document}
%-------------------------------------------------------------------------
% editorial commands: to be inserted by the editorial office
%
%\firstpage{1}
%\volume{228}
%\DOI{003-0001}
%\seriesextraline{This is the Concrete Title of this Book\br H.E. R and S.T.C. W, Eds.}
%\firstpage{1}
\title[OPS, random matrices, Painlev\'e equations]{Orthogonal and multiple orthogonal poly\-nomials, random matrices, and Painlev\'e equations}
%----------Author 1
\author{Walter Van Assche}

\address{%
Department of Mathematics\\
KU Leuven\\
Celestijnenlaan 200B box 2400\\
BE 3001 Leuven, Belgium}

\email{walter.vanassche@kuleuven.be}



\date{\today}

\begin{abstract}
Orthogonal polynomials and multiple orthogonal polynomials are interesting special functions because there is a beautiful theory for them, with many examples and useful applications in mathematical physics, numerical analysis,
statistics and probability and many other disciplines. In these notes we give an introduction to the use of orthogonal  polynomials in random matrix theory, we explain the notion of multiple orthogonal polynomials, and we show the link with certain non-linear difference and differential equations known as Painlev\'e equations.
\end{abstract}

%%% ----------------------------------------------------------------------
\maketitle
%%% ----------------------------------------------------------------------
\setcounter{tocdepth}{2}
\tableofcontents

\section{Introduction}
For these lecture notes I assume the reader is familiar with the basic theory of
orthogonal polynomials, in particular the classical orthogonal polynomials (Jacobi, Laguerre, Hermite) should be
known. In this introduction we will fix the notation and terminology. Let $\mu$ be a positive measure on the real line
for which all the moments $m_n$, $n \in \mathbb{N} = \{0,1,2,3,\ldots\}$ exist, where
\[       m_n = \int_{\mathbb{R}} x^n \, d\mu(x)   .   \]
The orthonormal polynomials $(p_n)_{n \in \mathbb{N}}$ are such that $p_n(x) = \gamma_n x^n + \cdots$, with $\gamma_n >0$, satisfying
the orthogonality condition  
\[   \int_{\mathbb{R}} p_n(x) p_m(x) \, d\mu(x) = \delta_{m,n}, \qquad m,n \in \mathbb{N}.  \]
It is well known that the zeros of $p_n$ are real and simple, and we denote them by 
\[  x_{1,n} < x_{2,n} < \cdots < x_{n,n} . \]
Orthonormal polynomials on the real line always satisfy a three-term recurrence relation
\begin{equation}   \label{3TRR}
    xp_n(x) = a_{n+1}p_{n+1}(x) + b_n p_n(x) +a_n p_{n-1}(x), \qquad n \geq 1, 
\end{equation}
with initial condition $p_0 = 1/\sqrt{m_0}$ and $p_{-1}=0$, with recurrence coefficients $a_{n+1} >0$ and $b_n \in \mathbb{R}$ for $n \geq 0$.
Often we will also use monic orthogonal polynomials, which we denote by capital letters:
\[    P_n(x) = \frac{1}{\gamma_n} p_n(x) = x^n + \cdots. \]
Their recurrence relation is of the form
\begin{equation}   \label{M3TRR}
    P_{n+1}(x) = (x-b_n) P_n(x) - a_n^2 P_{n-1}(x), 
\end{equation}
with initial conditions $P_0=1$ and $P_{-1}=0$. The classical families of orthogonal polynomials are
\begin{itemize}
  \item The \textit{Jacobi} polynomials $P_n^{(\alpha,\beta)}$, for which
  \[    \int_{-1}^1 P_n^{(\alpha,\beta)}(x) P_m^{(\alpha,\beta)}(x) (1-x)^\alpha (1+x)^\beta\, dx = 0, \qquad m \neq n,  \]
  with parameters $\alpha, \beta > -1$.
  \item The \textit{Laguerre polynomials} $L_n^{(\alpha)}$ for which
  \[    \int_0^\infty  L_n^{(\alpha)}(x)L_m^{(\alpha)}(x) x^\alpha e^{-x} \, dx = 0, \qquad m \neq n, \]
  with parameter $\alpha > -1$.
  \item The \textit{Hermite polynomials} $H_n(x)$ for which
  \[     \int_{-\infty}^\infty H_n(x) H_m(x) e^{-x^2}\, dx = 0, \qquad m \neq n.   \] 
\end{itemize}     		
Usually these polynomials are neither normalized nor monic but another normalization is used (for historical reasons) and one has to be a bit careful with
some of the general formulas for orthonormal or monic orthogonal polynomials.

The matrix 
\[   H_n = \begin{pmatrix}  m_0 & m_1 & m_2 & \cdots & m_{n-1} \\
                                  m_1 & m_2 & m_3 & \cdots & m_n \\
                                  m_2 & m_3 & m_4 & \cdots & m_{n+1} \\
                                  \vdots & \vdots & \vdots & \cdots & \vdots \\
                                  m_{n-1} & m_n & m_{n+1} & \cdots & m_{2n-2}
                      \end{pmatrix}  =  \bigl( m_{i+j-2} \bigr)_{i,j=1}^{n}  \]
is the \textit{Hankel matrix} with the moments of the orthogonality measure $\mu$. 
The \textit{Hankel determinant} is 
\begin{equation} \label{Hankeldet}
    D_n = \det \begin{pmatrix}  m_0 & m_1 & m_2 & \cdots & m_{n-1} \\
                                  m_1 & m_2 & m_3 & \cdots & m_n \\
                                  m_2 & m_3 & m_4 & \cdots & m_{n+1} \\
                                  \vdots & \vdots & \vdots & \cdots & \vdots \\
                                  m_{n-1} & m_n & m_{n+1} & \cdots & m_{2n-2}
                      \end{pmatrix}  = \det \bigl( m_{i+j-2} \bigr)_{i,j=1}^{n} .  
\end{equation}
If the support of $\mu$ contains infinitely many points, then
$D_n > 0$ for all $n \in \mathbb{N}$.

The monic orthogonal polynomials $P_n(x)$ are given by
\begin{equation}   \label{detP}
    P_n(x) = \frac{1}{D_n} \det  \begin{pmatrix}
                                       m_0 & m_1 & m_2 & \cdots & m_{n} \\
                                  m_1 & m_2 & m_3 & \cdots & m_{n+1} \\
                                  m_2 & m_3 & m_4 & \cdots & m_{n+2} \\
                                  \vdots & \vdots & \vdots & \cdots & \vdots \\
                                  m_{n-1} & m_n & m_{n+1} & \cdots & m_{2n-1} \\
                                  1 & x & x^2 & \cdots & x^n
                      \end{pmatrix},  
\end{equation}
and  
\begin{equation}   \label{gammaD}
   \frac{1}{\gamma_n^2} = \int_{\mathbb{R}} P_n^2(x)\, d\mu(x) = \frac{D_{n+1}}{D_n}.  
\end{equation}

The \textit{Christoffel-Darboux kernel} is defined as
\[   K_n(x,y) =  \sum_{k=0}^{n-1} \gamma_k^2 P_k(x)P_k(y) = \sum_{k=0}^{n-1} p_k(x)p_k(y)  . \]
This Christoffel-Darboux kernel is a reproducing kernel:
for every polynomial $q_{n-1}$ of degree $\leq n-1$ one has
\[    \int K_n(x,y) q_{n-1}(y) \, d\mu(y) = q_{n-1}(x).   \]
If $f$ is a function in $L^2(\mu)$, then
\[       \int K_n(x,y) f(y)\, d\mu(y) = f_{n-1}(x)  \]
gives a polynomial of degree $\leq n-1$ which is the least squares approximant of $f$ in the space of polynomials of degree $\leq n-1$.
The Christoffel-Darboux kernel is a sum of $n$ terms containing all the polynomials $p_0,p_1,\ldots,p_{n-1}$, but there is a nice formula that expresses the kernel
in just two terms containing the polynomials $p_{n-1}$ and $p_n$ only:
\begin{property}
The Christoffel-Darboux formula is
\[    \sum_{k=0}^{n-1} \gamma_k^2 P_k(x)P_k(y) = \gamma_{n-1}^2 \frac{P_n(x)P_{n-1}(y) - P_{n-1}(x)P_n(y)}{x-y} , \]
and its confluent version is
\[   \sum_{k=0}^{n-1} \gamma_k^2 P_k^2(x) = \gamma_{n-1}^2 \Bigl( P_n'(x)P_{n-1}(x) - P_{n-1}'(x)P_n(x) \Bigr). \]
\end{property}
The version for orthonormal polynomials is
\begin{property}
The Christoffel-Darboux formula is
\[    \sum_{k=0}^{n-1} p_k(x)p_k(y) = a_n \frac{p_n(x)p_{n-1}(y) - p_{n-1}(x)p_n(y)}{x-y} , \]
and its confluent version is
\[   \sum_{k=0}^{n-1}  p_k^2(x) = a_n \Bigl( p_n'(x)p_{n-1}(x) - p_{n-1}'(x)p_n(x) \Bigr). \]
\end{property}

\section{Orthogonal polynomials and random matrices}   \label{sec:RM}
The link between orthogonal polynomials and random matrices is via the Chris\-tof\-fel-Darboux kernel and Heine's formula for
orthogonal polynomials, see Property \ref{prop:2.1}. Useful references for random matrices are Mehta's book \cite{Mehta}, the book by Anderson, Guionnet and Zeitouni \cite{AGZ}, and Deift's monograph \cite{Deift}. 
First of all, let $x_1,x_2,\ldots,x_n$ be real or complex numbers, then we define the
\textit{Vandermonde determinant} as
\begin{equation}  \label{Vandermonde}
   \Delta_n(x_1,\ldots,x_n) = \det \begin{pmatrix}
                             1 & 1 & 1 & \cdots & 1 \\
                             x_1 & x_2 & x_3 & \cdots & x_n \\[4pt]
                             x_1^2 & x_2^2 & x_3^2 & \cdots & x_n^2 \\
                             \vdots & \vdots & \vdots & \cdots & \vdots \\
                             x_1^{n-1} & x_2^{n-1} & x_3^{n-1} & \cdots & x_n^{n-1} 
                                   \end{pmatrix}.  
\end{equation}
This Vandermonde determinant can be evaluated explicitly:
\[    \Delta_n = \prod_{i > j} (x_i-x_j) .  \]
From this it is clear that $\Delta_n \neq 0$ when all the $x_i$ are distinct, and if $x_1 < x_2 < \cdots < x_n$ then $\Delta_n >0$. 
Heine's formula expresses the Hankel determinant with the moments of a measure $\mu$ as an $n$-fold integral:

\begin{property}[Heine]   \label{prop:2.1}
The Hankel determinants $D_n$ in \eqref{Hankeldet} can be written as
\begin{equation}  \label{HeineD}
     D_n = \frac{1}{n!} \int_{-\infty}^\infty \cdots \int_{-\infty}^\infty  \Delta_n^2(x_1,\ldots,x_n)\, d\mu(x_1) \cdots d\mu(x_n) , 
\end{equation}
where $\Delta_n$ is the Vandermonde determinant \eqref{Vandermonde}. Furthermore, the monic orthogonal polynomial $P_n(x)$ is also given by an $n$-fold integral
\begin{equation}  \label{HeineP}
   P_n(x) = \frac{1}{n!D_n}  \int_{-\infty}^\infty \cdots \int_{-\infty}^\infty \prod_{i=1}^n (x-x_i)
  \  \Delta_n^2(x_1,\ldots,x_n) \, d\mu(x_1)\cdots d\mu(x_n) . 
\end{equation}
\end{property}

\begin{proof}
If we write all the moments in the first row of \eqref{Hankeldet} as an integral and use linearity of the determinant (for one row), then
\[   D_n = \int_{-\infty}^\infty \det \begin{pmatrix}  1 & x_1 & x_1^2 & \cdots & x_1^{n-1} \\
                                  m_1 & m_2 & m_3 & \cdots & m_n \\
                                  m_2 & m_3 & m_4 & \cdots & m_{n+1} \\
                                  \vdots & \vdots & \vdots & \cdots & \vdots \\
                                  m_{n-1} & m_n & m_{n+1} & \cdots & m_{2n-2}
                      \end{pmatrix} \, d\mu(x_1).  \]
Repeating this for every row gives
\[   D_n = \int_{-\infty}^\infty  \cdots \int_{-\infty}^\infty \det \begin{pmatrix}  1 & x_1 & x_1^2 & \cdots & x_1^{n-1} \\
                                  x_2 & x_2^2 & x_2^3 & \cdots & x_2^n \\
                                  x_3^2 & x_3^3 & x_3^4 & \cdots & x_3^{n+1} \\
                                  \vdots & \vdots & \vdots & \cdots & \vdots \\
                                  x_n^{n-1} & x_n^n & x_n^{n+1} & \cdots & x_n^{2n-2}
                      \end{pmatrix}\, d\mu(x_1)\cdots d\mu(x_n).  \]
In each row we can take out the common factors to find
\begin{equation*}
   D_n =  \int_{\mathbb{R}^n} \prod_{j=1}^n x_j^{j-1}  \, \Delta_n(x_1,\ldots,x_n)\, d\mu(x_1)\cdots d\mu(x_n).
\end{equation*}
Now write the integral over $\mathbb{R}^n$ as a sum of integrals over all simplices $x_{i_1} < x_{i_2} < \cdots < x_{i_n}$, where 
$\sigma=(i_1,i_2,\ldots,i_n)$ is a permutation of $(1,2,\ldots,n)$. Then
\[    D_n = \sum_{\sigma \in S_n} \int_{x_{\sigma(1)} < \cdots < x_{\sigma(n)}} \prod_{j=1}^n x_j^{j-1}  
\, \Delta_n(x_1,x_2,\ldots,x_n)\, d\mu(x_1)\cdots d\mu(x_n). \]
With the change of variables $x_{\sigma(j)}=y_j$ one has $x_j = y_{\tau(j)}$, with $\tau=\sigma^{-1}$ and
\[   D_n =  \int_{y_1 < \cdots < y_n} \sum_{\tau \in S_n} \prod_{j=1}^n y_{\tau(j)}^{j-1}  
\, \Delta_n(y_{\tau(1)},\ldots,y_{\tau(n)})\, d\mu(y_1)\cdots d\mu(y_n). \]
Observe that $\Delta_n(y_{\tau(1)},\ldots,y_{\tau(n)}) = \textup{sign}(\tau) \Delta_n(y_1,\ldots,y_n)$, so that
\[   D_n = \int_{y_1 < \cdots < y_n} \left( \sum_{\tau \in S_n} \textup{sign}(\tau) \prod_{j=1}^n y_{\tau(j)}^{j-1} \right) 
\, \Delta_n(y_1,\ldots,y_n)\, d\mu(y_1)\cdots d\mu(y_n). \]
Now use
\[    \sum_{\tau \in S_n} \textup{sign}(\tau) \prod_{j=1}^n y_{\tau(j)}^{j-1} = \Delta_n(y_1,\ldots,y_n)  \]
to find
\[   D_n = \int_{y_1 < \cdots < y_n}  \Delta_n^2(y_1,\ldots,y_n)\, d\mu(y_1)\cdots d\mu(y_n). \]
This is an integral over one simplex $y_1<y_2<\cdots<y_n$ in $\mathbb{R}^n$. This integral is the same for every simplex, and since there are $n!$ simplices (because there are $n!$
permutations of $(1,2,\ldots,n)$), we find the required formula \eqref{HeineD}.

The proof for formula \eqref{HeineP} is similar, using the determinant expression \eqref{detP} for the monic orthogonal polynomial. 
\end{proof}

\noindent It is remarkable that Szeg\H{o} writes in his book \cite{Szego}:
\begin{quote}
[These] Formulas \ldots\ are not suitable in general for derivation of properties of the polynomials in question. 
To this end we shall generally prefer the orthogonality property itself, or other representations derived by means of the orthogonality property. 
\end{quote}
Heine's formulas have now become crucial in the theory of random matrices.

\subsection{Point processes}

A $n$-point process is a stochastic process where a set of $n$ points $\{X_1,\ldots,X_n\}$ is selected, and the joint distribution of the
random variables $(X_1,X_2,\ldots,X_n)$ is given. Since we are dealing with a set of $n$ random numbers, the order of the random variables is irrelevant
and hence we use a probability distribution which is invariant under permutations.
Our interest is in the $n$-point process where the joint probability distribution has a density (with respect to the product measure $d\mu(x_1)\ldots d\mu(x_n)$) given by
\begin{equation}  \label{Ppoint}
     P(x_1,x_2,\ldots,x_n)  = \frac{1}{n! D_n} \Delta_n^2(x_1,\ldots,x_n) , 
\end{equation}
where we mean that 
\[   \textup{Prob}(X_1 \leq y_1,\ldots, X_n \leq y_n) = \int_{-\infty}^{y_1}\ldots \int_{-\infty}^{y_n} P(x_1,\ldots,x_n)\, d\mu(x_1)\cdots d\mu(x_n). \] 
Observe that by Heine's formula \eqref{HeineD} this is indeed a probability distribution since it is positive and integrates over $\mathbb{R}^n$ to one.
The points in this $n$-point process are not independent and the factor $\Delta_n^2(x_1,\ldots,x_n)$ describes the dependence of the points. Two points are unlikely to be close together
because then $\Delta_n^2(x_1,\ldots,x_n) = \prod_{j > i} (x_j-x_i)^2$ is small and by the maximum likelihood principle the points will prefer to choose a position that
maximizes $\Delta_n^2(x_1,\ldots,x_n)$. This $n$-point process therefore has points that repel each other.

An important property of this $n$-point process is that it is a \textit{determinantal point process}. To see this, we will express the probability density in terms of the
Christoffel-Darboux kernel. We need a few important properties of that kernel.
\begin{property}  \label{prop:CD}
The Christoffel-Darboux kernel satisfies
\[   \int_{-\infty}^\infty  K_n(x,y) K_n(y,z) \, d\mu(y) = K_n(x,z),  \]
and
\[     \int_{-\infty}^\infty  K_n(x,x) \, d\mu(x) = n .   \]
\end{property} 
\begin{proof}
The first property follows from the reproducing property of the Christoffel-Darboux kernel.
For the second property we have
\[   \int_{-\infty}^\infty K_n(x,x)\, d\mu(x) = \sum_{k=0}^{n-1} \int_{-\infty}^\infty p_k^2(x)\, d\mu(x) = n. \]
\end{proof}

\begin{property}
The density \eqref{Ppoint} can be written as
\[      P(x_1,x_2,\ldots,x_n) = \frac{1}{n!}\det \bigl(  K_n(x_i,x_j) \bigr)_{i,j=1}^n ,   \]
where $K_n$ is the Christoffel-Darboux kernel.
\end{property}
\begin{proof}
If we add rows in the Vandermonde determinant \eqref{Vandermonde}, then 
\[   \Delta_n(x_1,\ldots,x_n) = \det \begin{pmatrix}
                                     P_0(x_1) & P_0(x_2) & P_0(x_3) & \cdots & P_0(x_n) \\
                                     P_1(x_1) & P_1(x_2) & P_1(x_3) & \cdots & P_1(x_n) \\
                                     P_2(x_1) & P_2(x_2) & P_2(x_3) & \cdots & P_2(x_n) \\
                                      \vdots & \vdots & \vdots & \cdots & \vdots \\
                                     P_{n-1}(x_1) & P_{n-1}(x_2) & P_{n-1}(x_3) & \cdots & P_{n-1}(x_n) 
                                 \end{pmatrix},  \]
for any sequence $(P_0,P_1,P_2,\ldots,P_{n-1})$ of monic polynomials. If we take the monic orthogonal polynomials, then
\begin{multline*} \left( \prod_{j=0}^{n-1} \gamma_j^2 \right) \Delta_n^2(x_1,\ldots,x_n) \\
                               = {\small \det \begin{pmatrix}
                                     P_0(x_1) & P_1(x_1) & \cdots & P_{n-1}(x_1) \\
                                     P_0(x_2) & P_1(x_2) & \cdots & P_{n-1}(x_2) \\
                                     P_0(x_3) & P_1(x_3) & \cdots & P_{n-1}(x_3) \\
                                      \vdots & \vdots & \cdots & \vdots \\
                                     P_0(x_n) & P_1(x_n) & \cdots & P_{n-1}(x_n) 
                                 \end{pmatrix}
                                      \Gamma_n
                                     \begin{pmatrix}
                                     P_0(x_1) & P_0(x_2) & \cdots & P_0(x_n) \\
                                     P_1(x_1) & P_1(x_2) & \cdots & P_1(x_n) \\
                                     P_2(x_1) & P_2(x_2) & \cdots & P_2(x_n) \\
                                      \vdots & \vdots &  \cdots & \vdots \\
                                     P_{n-1}(x_1) & P_{n-1}(x_2) & \cdots & P_{n-1}(x_n) 
                                 \end{pmatrix},} 
\end{multline*}
where $\Gamma_n = \textup{diag}( \gamma_0^2, \gamma_1^2, \ldots, \gamma_{n-1}^2)$. Then use \eqref{gammaD} to find that $\prod_{j=0}^{n-1} \gamma_j^2 = 1/D_n$, so that
\[    \Delta_n^2(x_1,\ldots,x_n) = D_n \det \Bigl( \sum_{k=0}^{n-1} \gamma_k^2 P_k(x_i)P_k(x_j) \Bigr)_{i,j=1}^n,  \]
which combined with \eqref{Ppoint} gives the required result.
\end{proof}

For this reason we call the $n$-point process with density \eqref{Ppoint} the \textit{Christoffel-Darboux point process}.

\subsection{Determinantal point process}
The fact that the density $P(x_1,\ldots,x_n)$ can be written as a determinant of a kernel function $K(x,y)$ that satisfies
Property \ref{prop:CD} is important and allows to compute correlation functions for $k$ points $k \leq n$ of the point process,
in particular the probability density of one point (for $k=1$).

\begin{definition}
For $k \leq n$ the $k$th correlation function is 
\[   \rho_k(x_1,\dots,x_k) =\det \Bigl(  K_n(x_i,x_j) \Bigr)_{i,j=1}^k  .  \]
\end{definition}

The interpretation of these $k$th correlation functions is the following: if $A_i \cap A_j = \emptyset$ ($i \neq j$), and $N(A)$ is the number of points in $A$, then
\[   \int_{A_1} \int_{A_2} \cdots \int_{A_k} \rho_k(x_1,\ldots,x_k) \, d\mu(x_1)\cdots d\mu(x_k) =
   \mathbb{E}\left(\prod_{i=1}^k N(A_i) \right) .  \]
The $k$th correlation function can also be seen as the density of the marginal distribution of $k$ points in the $n$-point process, up to a normalization factor:

\begin{property}The $k$th correlation function is obtained from $P(x_1,\ldots,x_n)$ by
\begin{equation*}
    \rho_k(x_1,x_2,\ldots,x_k)     
= \frac{n!}{(n-k)!} \underbrace{\int_{-\infty}^\infty \cdots \int_{-\infty}^\infty}_{n-k} P(x_1,\ldots,x_n) \, d\mu(x_{k+1}) \cdots d\mu(x_n). 
 \end{equation*}
\end{property}

\begin{proof}
For $k=n-1$ we have, by expanding the determinant along the last row,
\begin{multline*}   
  \int_{-\infty}^\infty P(x_1,\ldots,x_n)\, d\mu(x_n) \\
  = \frac{1}{n!} \sum_{k=1}^{n-1} \int_{-\infty}^\infty (-1)^{n+k} K_n(x_n,x_k) \det \Bigl( K_n(x_i,x_j) \Bigr)_{1\leq i\neq n,j\neq k \leq n} \, d\mu(x_n) \\
   + \frac{1}{n!} \int_{-\infty}^\infty K_n(x_n,x_n) \det \Bigl( K_n(x_i,x_j) \Bigr)_{i,j=1}^{n-1} \, d\mu(x_n).  
\end{multline*}
By Property \ref{prop:CD} the last term is $1/(n-1)! \rho_{n-1}(x_1,\ldots,x_{n-1})$. Expanding the remaining determinant along the last column gives
\begin{multline*}
   \frac{1}{n!} \sum_{k=1}^{n-1} \sum_{\ell=1}^{n-1} (-1)^{n+k} (-1)^{n-1+\ell} \int_{-\infty}^\infty K_n(x_n,x_k) K_n(x_\ell,x_n) \\
  \times  \det \Bigr( K_n(x_i,x_j) \Bigr)_{1 \leq i\neq\ell ,j\neq k\leq n-1}\,
    d\mu(x_n) .  
\end{multline*}
The determinant does not contain $x_n$, so the remaining integration can be done using Property \ref{prop:CD} and gives
\[   \frac{1}{n!} \sum_{k=1}^{n-1} \sum_{\ell=1}^{n-1} (-1)^{k+\ell-1} K_n(x_\ell,x_k)  \det \Bigr( K_n(x_i,x_j) \Bigr)_{1 \leq i\neq\ell ,j\neq k\leq n-1}.  \]
The sum over $\ell$ gives the $(n-1)\times(n-1)$ determinant (recall that column $k$ which contains $K_n(x_i,x_k)$ is missing since $j\neq k$)
\[    (-1)^n \det \begin{pmatrix}  K_n(x_1,x_1) & K_n(x_1,x_2) & \cdots & K_n(x_1,x_{n-1}) & K_n(x_1,x_k) \\
                                   K_n(x_2,x_1) & K_n(x_2,x_2) & \cdots & K_n(x_2,x_{n-1}) & K_n(x_2,x_k) \\
                                        \vdots & \vdots & \cdots & \vdots & \vdots \\
                                   K_n(x_{n-1},x_1) & K_n(x_{n-1},x_2) & \cdots & K_n(x_{n-1},x_{n-1}) & K_n(x_{n-1},x_k)
                      \end{pmatrix}, \]
and to get the last column in the $k$th position, we need to interchange columns $n-1-k$ times, which gives          
\begin{multline*}   \int_{-\infty}^\infty P(x_1,\ldots,x_n)\, d\mu(x_n) \\
  = \frac{-1}{n!} \sum_{k=1}^{n-1} \rho_{n-1}(x_1,\dots,x_{n-1}) + \frac{1}{(n-1)!} \rho_{n-1}(x_1,\ldots,x_{n-1}),   
\end{multline*}
and hence
\[  \rho_{n-1}(x_1,\ldots,x_{n-1}) = n! \int_{-\infty}^\infty P(x_1,\ldots,x_n)\, d\mu(x_n).  \]
To prove the case for all $k=n-m$ one uses induction on $m$, for which we just proved the case $m=1$.
\end{proof}

\begin{definition}
A point process on $\mathbb{R}$ with correlation functions $\rho_k$ is a \textit{determinantal point process}
if there exists a kernel $K(x,y)$ such that for every $k$ and every $x_1,\ldots,x_k \in \mathbb{R}$
\[   \rho_k(x_1,x_2,\ldots,x_k) = \det \bigl( K(x_i,x_j)  \bigr)_{i,j=1}^k.  \]
\end{definition}

The following theorem shows that Property \ref{prop:CD} is indeed crucial. 

\begin{theorem}   \label{thm:detproc}
Suppose $K: \mathbb{R} \times \mathbb{R} \to \mathbb{R}$ is a kernel such that
\begin{itemize}
\item  $ \int_{-\infty}^\infty K(x,x)\, dx = n \in \mathbb{N}$,
\item For every $x_1,\ldots,x_n \in \mathbb{R}$, one has $\det \bigl( K(x_i,x_j) \bigr)_{i,j=1}^k \geq 0$.
\item $K(x,y) = \int_{-\infty}^\infty  K(x,s) K(s,y)\, ds $.
\end{itemize}
Then 
\[   P(x_1,\ldots,x_n) = \frac{1}{n!} \det \bigl(  K(x_i,x_j) \bigr)_{i,j=1}^n  \]
is a probability density on $\mathbb{R}^n$ which is invariant under permutations of coordinates. The associated $n$-point process is determinantal.
 \end{theorem}
 
The most important example (at least in the context of this section) is when $d\mu(x)=w(x)\, dx$, and then one can take
 \[   K(x,y) = K_n(x,y) \sqrt{w(x)} \sqrt{w(y)}.  \]
 
\subsection{Random matrices}
To see the relation with random matrices, we claim that the eigenvalues of certain random matrices of order $n$ form a determinantal
point process with the Christoffel-Darboux kernel for a particular family of orthogonal polynomials. 
The \textit{Gaussian unitary ensemble} (GUE) consists of Hermitian random matrices $\mathbf{M}$ of order $n$ with random entries
\[   \mathbf{M}_{k,\ell} = X_{k,\ell} + i Y_{k,\ell}, \quad 
   \mathbf{M}_{\ell,k} = X_{k,\ell} - i Y_{k,\ell},  \qquad k < \ell, \]
\[     \mathbf{M}_{k,k} = X_{k,k}, \qquad 1 \leq k \leq n,  \]
where all $X_{k,\ell}, Y_{k,\ell}, X_{k,k}$ are independent normal random variables with mean zero and variance $\frac{1}{4n}$ (if $k < \ell$)
or $\frac{1}{2n}$ (if $k=\ell$).
The multivariate density is
\[    \frac{1}{Z_n} \prod_{k < \ell}  e^{-2n(x_{k,\ell}^2+y_{k,\ell}^2)} \prod_{k=1}^n e^{-nx_{k,k}^2} \, \prod_{k<\ell} dx_{k,\ell}dy_{k,\ell} \prod_{k=1}^n dx_{k,k}, \]
where $Z_n$ is normalizing constant. But this is also equal to
\[   \frac{1}{Z_n} \exp( -n \textup{Tr}\ M^2 ) \, dM  \]
where $M_{k,\ell} = (x_{k,\ell}+iy_{k,\ell})$ for $k < \ell$, $M_{k,k} = x_{k,k}$, and $M=M^*$.

We are mostly interested in the eigenvalues $\lambda_1,\ldots,\lambda_n$ of the random matrix $\mathbf{M}$.
To find the density of the eigenvalues, we use the change of variables: ${M} \mapsto (\Lambda,U)$, where $U$ is a unitary matrix for which
\[     {M} = U \Lambda U^*, \]
and $\Lambda= \textup{diag} (\lambda_1,\ldots,\lambda_n)$, and then integrate over the unitary part $U$, which leaves only the eigenvalues. 
This change of variables is done using the Weyl integration formula (see, e.g., \cite[\S 4.1.3]{AGZ}):

\begin{theorem}[Weyl integration formula]
For the change of variables $M = U \Lambda U^*$ one has
\[    dM = c_n \prod_{i<j} (\lambda_i-\lambda_j)^2 \, d\lambda_1\cdots d\lambda_n \, dU,  \]
where $c_n$ is a constant and $dU$ is the Haar measure on the unitary group.
\end{theorem}

We will use a simplified version of this result, for which one does not need the Haar measure on the unitary group.
This works when the expression $f(M)$ that we want to integrate only depends on the eigenvalues of $M$.
Let $\mathcal{H}_n$ be the Hermitian matrices of order $n$.
\begin{definition}
A function $f: \mathcal{H}_n \to \mathbb{C}$  is a \textit{class function} if
\[    f(UMU^*) = f(M) \]
for all unitary matrices $U$.
\end{definition}

\begin{theorem}[Weyl integration formula for class functions]     \label{thm:WeylC}
For an integrable class function $f$ we have
\[  \int f(M)\, dM = c_n  \int_{\mathbb{R}^n} f(\lambda_1,\ldots,\lambda_n) \prod_{i<j} (\lambda_i-\lambda_j)^2
\, d\lambda_1 \cdots d\lambda_n, \]
with
\[   c_n = \frac{\pi^{n(n-1)/2}}{\prod_{j=1}^n j! } .  \]
\end{theorem}

The characteristic polynomial of a matrix $M$ only depends on the eigenvalues, hence $\det(xI-M)$ is a class function.
For random matrices in GUE one finds for the average characteristic function
\begin{equation} \label{avcharpol}
  \mathbb{E} \det ( x I-\mathbf{M}) 
  = \frac{1}{D_n} \int_{\mathbb{R}^n} \prod_{i=1}^n (x-x_i)
      \ \Delta_n^2(x_1,\ldots,x_n) e^{-n(x_1^2+\cdots+x_n^2)}\, dx_1\cdots dx_n 
\end{equation}
and by \eqref{HeineP} this is the monic \textit{Hermite polynomial} $H_n(\sqrt{n}x)$.
More generally, \textbf{the eigenvalues of a random matrix in GUE form a determinantal point process with the Christoffel-Darboux kernel
of (scaled) Hermite polynomials}.
The average number of eigenvalues of $\mathbf{M}$ in $[a,b]$ is in terms of the correlation function $\rho_1(x)$:
\[   \mathbb{E}\bigl(N([a,b])\bigr) =  \int_a^b K_n(x,x)\, e^{-nx^2}\, dx.   \]

\subsection{Random matrix ensembles}
Here we give a few more random matrix ensembles for which the eigenvalues form a determinantal point process with the Christoffel-Darboux kernel of classical orthogonal polynomials.
 \begin{itemize}
 \item We already defined GUE (Gaussian Unitary Ensemble): this contains random matrices in $\mathcal{H}_n$ with density
 \[     \frac{1}{Z_n}  \exp(-n \textup{Tr }M^2)\, dM.   \] 
The average characteristic polynomial is
\[    \mathbb{E}\det (x I - \mathbf{M}) = \textrm{(scaled) Hermite polynomial} . \]
This suggests that on the average the eigenvalues behave like the zeros of (scaled) Hermite polynomials. This is indeed true, but for this one needs the correlation function
$\rho_1$ and the result that
\[     \lim_{n \to \infty} \frac{1}{n} \int_a^b f(x) K_n(x,x) e^{-nx^2} \, dx = \lim_{n \to \infty} \frac{1}{n} \sum_{j=1}^n f(x_{j,n}/\sqrt{n}), \]
where $x_{1,n},\ldots,x_{n,n}$ are the zeros of the Hermite polynomial $H_n$. 
\item The Wishart ensemble. Let $\mathbf{M}$ be a $n \times m$ matrix $(m \geq n)$ with independent complex Gaussian entries $X_{k,\ell} + i Y_{k,\ell}$. 
Then $\mathbf{M}\mathbf{M}^*$ has the Wishart distribution with density
\[ \frac{1}{C_n}  |\det W|^{m-n} \exp ( - \textrm{Tr } W  ).  \]
The average characteristic polynomial is
\[    \mathbb{E} \det ( x I - \mathbf{M}\mathbf{M}^*) = \textrm{ Laguerre polynomial with } \alpha=m-n.   \]
Observe that $\mathbf{M}\mathbf{M}^*$ is a positive definite matrix so that all the eigenvalues are positive. On the average they behave like the zeros of
Laguerre polynomials.
\item Truncated unitary matrices. Let $U$ be a random unitary matrix of order $(m+k)\times(m+k)$ and let $\mathbf{V}$ be the $m\times n$ upper left
 corner $(m \geq n)$.
Then $\mathbf{V}^*\mathbf{V}$ is an $n\times n$ matrix and
\[ \quad  \mathbb{E} \det (x I - \mathbf{V}^*\mathbf{V}) = \textrm{Jacobi polynomial on } [0,1], \quad \alpha=m-n,\beta=k-n. \]
Unitary matrices have their eigenvalues on the unit circle, and a truncated unitary matrix has its singular values (the eigenvalues of $\mathbf{V}^*\mathbf{V}$) in $[0,1]$.
These eigenvalues behave on the average like the zeros of Jacobi polynomials.
\end{itemize}
\medskip

\noindent\shadowbox{\parbox{12cm}{
\begin{exercise}
Let $\mathbf{M}_n$ be the Hermitian random matrix with entries
\[   (\mathbf{M}_n)_{k,\ell} = \begin{cases}  X_{k,\ell} + i Y_{k,\ell}, & k < \ell, \\
                                              X_{\ell,k} - i Y_{\ell,k}, & k > \ell, \\
                                              X_{k,k}, & k=\ell,
                               \end{cases}  \]
where $X_{k,\ell}, Y_{k,\ell}$ $(k < \ell)$ and $X_{k,k}$ $(1 \leq k \leq n)$ are independent random variables with means
$\mathbb{E}(X_{k,\ell}) = \mathbb{E}(Y_{k,\ell}) = \mathbb{E}(X_{k,k}) = 0$ and variances $\mathbb{E}(X_{k,\ell}^2)=\mathbb{E}(Y_{k,\ell}^2)=\mathbb{E}(X_{k,k}^2)= \sigma^2 >0$. Show that $P_n(x) = \mathbb{E} \det (xI_n - \mathbf{M}_n)$ satisfies the three-term recurrence relation
\[  P_n(x) = xP_{n-1}(x) - 2(n-1)\sigma^2 P_{n-2}(x), \]
with $P_0(x)=1$ and $P_1(x)=x$. Identify this $P_n(x)$ as $\sigma^n H_n(x/2\sigma)$, where $H_n$ is the Hermite polynomial of degree $n$.
This shows that the Hermite polynomial is the average characteristic polynomial of a large class of Hermitian random matrices, not only GUE.      
\end{exercise}}}

So far we found that on the average the eigenvalues of random matrices from these ensembles behave like zeros of orthogonal polynomials.
To get more information about individual eigenvalues, for example the largest eigenvalue or the smallest eigenvalue, one needs a more detailed analysis of the point process.
In particular one needs to investigate the asymptotic behavior of the Christoffel-Darboux kernels. In particular, to understand the spacing between the eigenvalues
in the neighborhood of $x^*$ in the bulk of the spectrum, one needs results for
\[   \lim_{n \to \infty} \frac{1}{n} K_n\bigl(x^*+\frac{u}{n},x^*+\frac{v}{n}\bigr), \]
or, when $x^*$ is at the end of the spectrum,   
\[   \lim_{n\to \infty} \frac{1}{n^\gamma} K_n\bigl(x^*+\frac{u}{n^\gamma},x^*+\frac{v}{n^\gamma}\bigr),  \]
where $\gamma$ depends on the nature of the endpoint (hard or soft edge).
This will give kernels of well-known point processes. 

An important quantity of interest is the probability $p_A(m)$ that there are exactly $m$ eigenvalues in the set $A \subset \mathbb{R}$.
If there are $m$  eigenvalues in $A$, then the number of ordered $k$-tuples in $A$ is $\binom{m}{k}$ and thus
\[   \sum_{m=k}^\infty \binom{m}{k} p_A(m) = \frac{1}{k!} \int_{A^k} \rho_k(x_1,\ldots,x_k)\, d\mu(x_1)\cdots d\mu(x_k), \qquad k \geq 1, \]
because this is the expected number of ordered $k$-tuples in $A$. For $k=0$ one has
\[   \sum_{m=0}^\infty p_A(m) = 1, \]
therefore
\[  1 + \sum_{k=1}^\infty \frac{(-1)^k}{k!} \int_{A^k} \rho_k(x_1,\ldots,x_k)\, d\mu(x_1)\cdots d\mu(x_k) =
      \sum_{k=0}^\infty \sum_{m=k}^\infty (-1)^k \binom{m}{k} p_A(m) .  \]
Changing the order of summation (we assume that this is allowed) and using
\[     \sum_{k=0}^m (-1)^k \binom{m}{k} = \delta_{m,0}, \]
we find that
\[    p_A(0) = 1 + \sum_{k=1}^\infty \frac{(-1)^k}{k!} \int_{A^k} \rho_k(x_1,\ldots,x_k)\, d\mu(x_1)\cdots d\mu(x_k).  \]
This is the so-called \textit{gap probability}: the probability to find no eigenvalues in $A$.
For a determinantal point process, such as the eigenvalues of various random matrices, this gap probability is in fact the
Fredholm determinant $\det ( I - K_A)$ of the operator $K_A: L^2(A)\to L^2(A)$ defined by
\[    K_Af(x) = \int_A K_n(x,y) f(y)\, d\mu(y), \qquad x \in A. \]
The asymptotic behavior as the size $n$ of the random matrices increases to infinity, then gives the Fredholm determinant $\det (I-K_A)$ of the 
operator $K_A$ that uses the kernel $K(x,y)$ which is the limit of the Christoffel-Darboux kernel $K_n(x,y)$ as described above.
The lesson to be learned from this is that the asymptotic behavior of orthogonal polynomials and their
Christoffel-Darboux kernel gives important insight in the behavior of eigenvalues of random matrices.

\section{Multiple orthogonal polynomials}

In this section we will explain the notion of multiple orthogonal polynomials. Useful references are Ismail's book \cite[Ch.~23]{Ismail}, Nikishin and Sorokin's book \cite[Ch.~4]{NikiSor}
and the papers \cite{Apt,MarWVA,RHP}.
Instead of orthogonality conditions with respect to one measure on the real line,
the orthogonality will be with respect to $r$ measures, where $r \geq 1$. For $r=1$ one has the usual orthogonal polynomials, but for $r\geq 2$ one gets
two types of multiple orthogonal polynomials. 
 
Let $r \in \mathbb{N}$ and let $\mu_1,\ldots,\mu_r$ be positive measures on the real line, for which all the moments exist. 
We use \textit{multi-indices} $\vec{n} = (n_1,n_2,\ldots,n_r) \in \mathbb{N}^r$ and denote their \textit{length} by
$|\vec{n}| = n_1+n_2+\cdots+n_r$. 

\begin{definition}[type I]
Type I multiple orthogonal polynomials for $\vec{n}$ consist of the vector $(A_{\vec{n},1},\ldots,A_{\vec{n},r})$
of $r$ polynomials, with $\deg A_{\vec{n},j} \leq n_j-1$, for which
\[   \int x^k \sum_{j=1}^r A_{\vec{n},j}(x) \, d\mu_j(x) = 0, \qquad 0 \leq k \leq |\vec{n}|-2, \]
with normalization
\[  \int x^{|\vec{n}|-1} \sum_{j=1}^r A_{\vec{n},j}(x) \, d\mu_j(x) = 1.  \]
\end{definition}

\begin{definition}[type II]
The type II multiple orthogonal polynomial for $\vec{n}$ is the \textbf{monic} polynomial $P_{\vec{n}}$ of degree $|\vec{n}|$ for which
\[   \int x^k P_{\vec{n}}(x)\, d\mu_j(x) = 0, \qquad 0 \leq k \leq n_j-1, \]
for $1 \leq j \leq r$.
\end{definition}

The conditions for type I and type II multiple orthogonal polynomials give a system of $|\vec{n}|$ linear equations for the $|\vec{n}|$ unknown coefficients
of the polynomials. This system may not have a solution, or when a solution exists it may not be unique. 
A multi-index $\vec{n}$ is said to be \textit{normal} if the type I vector $(A_{\vec{n},1}, \ldots A_{\vec{n},r})$ exists and is unique,
and this is equivalent with the existence and uniqueness of the monic
type II multiple orthogonal polynomial $P_{\vec{n}}$, because the matrix of the linear system for type II is the transpose of the matrix for the type I linear system.
Hence $\vec{n}$ is a normal multi-index if and only if
\[ \small   \det \begin{pmatrix}  M_{n_1}^{(1)} \\ M_{n_2}^{(2)} \\ \vdots \\ M_{n_r}^{(r)} \end{pmatrix} \neq 0, \]
where
\[         M_{n_j}^{(j)} =  \begin{pmatrix}
               m_0^{(j)} & m_1^{(j)} & \cdots & m_{|\vec{n}|-1}^{(j)} \\
                m_1^{(j)} & m_2^{(j)} & \cdots & m_{|\vec{n}|}^{(j)} \\
               \vdots & \vdots & \cdots & \vdots \\
                m_{n_j-1}^{(j)} & m_{n_j}^{(j)} & \cdots & m_{|\vec{n}|+n_j-2}^{(j)} 
             \end{pmatrix}   \]
are rectangular Hankel matrices containing the moments
\[     m_k^{(j)} = \int x^k\, d\mu_j(x).   \]

\subsection{Special systems}
Interesting systems of measures $(\mu_1,\ldots,\mu_r)$ are those for which all the multi-indices are normal. We call such systems \textit{perfect}.
Here we will describe two such systems.

\begin{definition}[Angelesco system]
The measures $(\mu_1,\ldots,\mu_r)$ are an \textit{Angelesco system} if the supports of the measures 
are subsets of disjoint intervals $\Delta_j$, i.e., $\textup{supp}(\mu_j) \subset \Delta_j$ and
$\Delta_i \cap \Delta_j = \emptyset$ whenever $i \neq j$.
\end{definition}

\noindent Usually one allows that the intervals are touching, i.e., $\stackrel{\circ}{\Delta_i} \cap \stackrel{\circ}{\Delta_j} = \emptyset$ whenever 
$i \neq j$.

\begin{theorem}[Angelesco, Nikishin]
The type II multiple orthogonal polynomial $P_{\vec{n}}$ for an Angelesco system has exactly $n_j$ distinct zeros on $\stackrel{\circ}{\Delta_j}$ for $1 \leq j \leq r$.
\end{theorem}

This means that the type II multiple orthogonal polynomial $P_{\vec{n}}$ can be factored as $P_{\vec{n}}(x) = \prod_{j=1}^r p_{\vec{n},j}(x)$, where $p_{\vec{n},j}$ has all its zeros
on $\Delta_j$. In fact, $p_{\vec{n},j}$ is an ordinary orthogonal polynomial of degree $n_j$ on the interval $\Delta_j$ for the measure
$\prod_{i\neq j} p_{\vec{n},i}(x)\ d\mu_j(x)$:
\[     \int_{\Delta_j} x^k  p_{\vec{n},j}(x) \ \prod_{i\neq j} p_{\vec{n},i}\ d\mu(x) = 0, \qquad  0 \leq k \leq n_j-1.  \]
Observe that for $i \neq j$ the polynomial $p_{\vec{n},i}(x)$ has constant sign on $\Delta_j$.

\begin{corollary}
Every multi-index $\vec{n}$ is normal (an Angelesco system is perfect).
\end{corollary}

\noindent\shadowbox{\parbox{12cm}{
\begin{exercise}
Show that every $A_{\vec{n},j}$ has $n_j-1$ zeros on $\stackrel{\circ}{\Delta_j}$.
\end{exercise}}}

For another system of measures, which are all supported on the same interval $[a,b]$, we need to recall the notion of a Chebyshev system.

\begin{definition}
The functions $\varphi_1,\ldots,\varphi_n$ are a \textbf{Chebyshev system} on $[a,b]$
if every linear combination $\sum_{i=1}^n a_i \varphi_i$ with $(a_1,\ldots,a_n) \neq (0,\ldots,0)$ has at most
$n-1$ zeros on $[a,b]$.
\end{definition}

We can then define an \textit{Algebraic Chebyshev system}:

\begin{definition}[AT-system]
The measures $(\mu_1,\ldots,\mu_r)$ are an \textit{AT-system} on the interval $[a,b]$ if the measures are all absolutely continuous
with respect to a positive measure $\mu$ on $[a,b]$, i.e., $d\mu_j(x) = w_j(x)\, d\mu(x)$ $(1 \leq j \leq r)$, and for every $\vec{n}$
the functions
\begin{multline*} w_1(x), xw_1(x), \ldots, x^{n_1-1}w_1(x),\ w_2(x), x w_2(x), \ldots, x^{n_2-1} w_2(x),  \\ 
   \ldots,    w_r(x), xw_r(x), \ldots, x^{n_r-1} w_r(x) 
\end{multline*}
are a Chebyshev system on $[a,b]$.
\end{definition}

For an AT-system we have some control of the zeros of the type I and type II multiple orthogonal polynomials. 

\begin{theorem}
For an AT-system the function 
\[ Q_{\vec{n}}(x) = \sum_{j=1}^r A_{\vec{n},j}(x) w_j(x)  \]
has exactly $|\vec{n}|-1$ sign changes on $(a,b)$.  
Furthermore, the type II multiple orthogonal polynomial $P_{\vec{n}}$ 
has exactly $|\vec{n}|$ distinct zeros on $(a,b)$.
\end{theorem}

\begin{corollary}
Every multi-index in an AT-system is normal (an AT-system is perfect).
\end{corollary}

A very special system of measures was introduced by Nikishin in 1980.

\begin{definition}[Nikishin system for $r=2$]
A \textit{Nikishin system} of order $r=2$ consists of two measures $(\mu_1,\mu_2)$, both supported on an interval $\Delta_2$,
and such that
\[    \frac{d\mu_2(x)}{d\mu_1(x)} = \int_{\Delta_1} \frac{d\sigma(t)}{x-t}, \]
where $\sigma$ is a positive measure on an interval $\Delta_1$ and $\Delta_1 \cap \Delta_2 = \emptyset$. 
\end{definition}

Nikishin showed that indices with $n_1 \geq n_2$ are perfect. Driver and Stahl \cite{DriverStahl} proved the more general statement.

\begin{theorem}[Nikishin, Driver-Stahl]
A Nikishin system of order two is perfect.
\end{theorem}

In order to define a Nikishin system of order $r >2$ we need some notation.
We write $\langle \sigma_1,\sigma_2 \rangle$ for the measure which is absolutely continuous with respect to $\sigma_1$ and
for which the Radon-Nikodym derivative is the Stieltjes transform of $\sigma_2$:
\[    d\langle \sigma_1,\sigma_2 \rangle(x) = \left(  \int \frac{d\sigma_2(t)}{x-t} \right) \, d\sigma_1(x). \]
Nikishin systems of order $r$ can then be defined by induction.

\begin{definition}[Nikishin system for general $r$]
A \textit{Nikishin system} of order $r$  on an interval $\Delta_r$ is a system of $r$ measures $(\mu_1,\mu_2,\ldots,\mu_r)$  supported on $\Delta_r$
such that $\mu_j = \langle \mu_1,\sigma_j\rangle$, $2 \leq j \leq r$, where $(\sigma_2,\ldots,\sigma_r)$ is a Nikishin system of
order $r-1$ on an interval $\Delta_{r-1}$ and $\Delta_r \cap \Delta_{r-1} = \emptyset$.
\end{definition}

\noindent Fidalgo Prieto and L\'opez Lagomasino proved \cite{FidLop} 

\begin{theorem}
Every Nikishin system is perfect.
\end{theorem}

In most cases the measures $(\mu_1,\ldots,\mu_r)$ are absolutely continuous with respect to one fixed measure $\mu$:
\[     d\mu_j(x) = w_j(x)\, d\mu(x), \qquad 1 \leq j \leq r.  \]
We then define the \textit{type I function}
\[     Q_{\vec{n}}(x) = \sum_{j=1}^r A_{\vec{n},j}(x) w_j(x).  \]
The type I functions and the type II polynomials then are very complementary: they form a biorthogonal system for many multi-indices.

\begin{property}[biorthogonality]   \label{prop:biorth}
\[  \int  P_{\vec{n}}(x)Q_{\vec{m}}(x)\, d\mu(x) =
    \begin{cases} 0, & \textup{if } \vec{m} \leq \vec{n}, \\
                  0, & \textup{if } |\vec{n}| \leq |\vec{m}| -2, \\
                  1, & \textup{if } |\vec{n}|=|\vec{m}|-1.  
    \end{cases}     \]
\end{property}

\subsection{Nearest neighbor recurrence relations} 
The usual orthogonal polynomials (the case $r=1$) on the real line always satisfy a three-term recurrence relation that expresses $xp_n(x)$ in terms of the
polynomials with neighboring degrees $p_{n+1}, p_n, p_{n-1}$. A similar result is true for multiple orthogonal polynomials, but there are more neighbors
for a multi-index. Indeed, the multi-index $\vec{n}$ has $r$ neighbors from above by adding 1 to one of the components of $\vec{n}$. We denote these
neighbors from above by $\vec{n}+\vec{e}_k$ for $1 \leq k \leq r$, where $\vec{e}_k = (0,\ldots,0,1,0,\ldots,0)$ with $1$ in position $k$. There are also $r$ neighbors from below, namely
$\vec{n}-\vec{e}_j$, for $1 \leq j \leq r$.
The nearest neighbor recurrence relations for type II multiple orthogonal polynomials are \cite{WVA2}
\begin{eqnarray*}
   xP_{\vec{n}}(x) &=&	 P_{\vec{n}+\vec{e}_1}(x) + b_{\vec{n},1}P_{\vec{n}}(x) + \sum_{j=1}^r a_{\vec{n},j} P_{\vec{n}-\vec{e}_j}(x), \\[-10pt]
                   &\vdots \\
    xP_{\vec{n}}(x) &=&	 P_{\vec{n}+\vec{e}_r}(x) + b_{\vec{n},r}P_{\vec{n}}(x) + \sum_{j=1}^r a_{\vec{n},j} P_{\vec{n}-\vec{e}_j}(x).
\end{eqnarray*}
Observe that one always uses the same linear combination of the neighbors from below.
The nearest neighbor recurrence relations for type I multiple orthogonal polynomials are  
\begin{eqnarray*}
   xQ_{\vec{n}}(x) &=& Q_{\vec{n}-\vec{e}_1}(x) + b_{\vec{n}-\vec{e}_1,1}Q_{\vec{n}}(x) + \sum_{j=1}^r a_{\vec{n},j} Q_{\vec{n}+\vec{e}_j}(x), \\[-10pt]
                   &\vdots \\
    xQ_{\vec{n}}(x) &=& Q_{\vec{n}-\vec{e}_r}(x) + b_{\vec{n}-\vec{e}_r,r}Q_{\vec{n}}(x) + \sum_{j=1}^r a_{\vec{n},j} Q_{\vec{n}+\vec{e}_j}(x).
\end{eqnarray*}
These are using the same recurrence coefficients $a_{\vec{n},j}$, but there is a shift for the recurrence coefficients $b_{\vec{n},k}$.
For $r \geq 2$ the recurrence coefficients $\{a_{\vec{n},j}, 1 \leq j \leq r\}$ and $\{b_{\vec{n},k}, 1 \leq k \leq r\}$ are connected:

\begin{theorem}[Van Assche \cite{WVA2}]  \label{thm:multcomp}
The recurrence coefficients $(a_{\vec{n},1}, \ldots,a_{\vec{n},r})$ and $(b_{\vec{n},1},\ldots,b_{\vec{n},r})$ 
satisfy the partial difference equations
\begin{eqnarray*}
      b_{\vec{n}+\vec{e}_i,j}-b_{\vec{n},j} & = & b_{\vec{n}+\vec{e}_j,i} - b_{\vec{n},i}, \\
     \sum_{k=1}^r a_{\vec{n}+\vec{e}_j,k} - \sum_{k=1}^r a_{\vec{n}+\vec{e}_i,k} & = & \det \begin{pmatrix}
                                                                 b_{\vec{n}+\vec{e}_j,i} & b_{\vec{n},i} \\
                                                                 b_{\vec{n}+\vec{e}_i,j} & b_{\vec{n},j}  \end{pmatrix}, \\
      \frac{a_{\vec{n},i}}{a_{\vec{n}+\vec{e}_j,i}}  &=&  \frac{b_{\vec{n}-\vec{e}_i,j} - b_{\vec{n}-\vec{e}_i,i}}{b_{\vec{n},j}-b_{\vec{n},i}}, 
\end{eqnarray*} 
for all $1 \leq i \neq j \leq r$.
\end{theorem}

By combining the equations of the nearest neighbor recurrence relations, one can also find a recurrence relation of order $r+1$ for
the multiple orthogonal polynomials along a path from $\vec{0}$ to $\vec{n}$ in $\mathbb{N}^r$.
Let $(\vec{n}_k)_{k \geq 0}$ be a path in $\mathbb{N}^r$ starting from $\vec{n}_0=\vec{0}$, 
such that $\vec{n}_{k+1} - \vec{n}_k = \vec{e}_i$ for some $1 \leq i \leq r$. Then
\[   xP_{\vec{n}_k}(x) = P_{\vec{n}_{k+1}}(x) + \sum_{j=0}^r \beta_{\vec{n}_k,j} P_{\vec{n}_{k-j}}(x).  \]
These $\beta_{\vec{n}_k,j}$ coefficients can be expressed in terms of the recurrence coefficients in the nearest neighbor recurrence relations,
but the explicit expression is rather complicated for general $r$.
An important case is the \textit{stepline}:
\[    \vec{n}_k = (\overbrace{i+1,\ldots,i+1}^j, \underbrace{i,\ldots i}_{r-j}), \qquad k = ri + j, \ 0 \leq j \leq r-1.  \]
This recurrence relation of order $r+1$ can be expressed in terms of a Hessenberg matrix with $r$ diagonals below the main diagonal:
\[  \small  x \begin{pmatrix} P_{\vec{n}_0}(x) \\ P_{\vec{n}_1}(x) \\ P_{\vec{n}_2}(x) \\ \vdots  \\ P_{\vec{n}_k}(x) \\ \vdots \end{pmatrix} \rule{4in}{0pt}  \] 
\[  \small = \begin{pmatrix}  \beta_{\vec{n}_0,0} & 1 & 0 & 0 & 0 &  0 &\cdots \\[2pt]
                       \beta_{\vec{n}_1,1} & \beta_{\vec{n}_1,0} & 1 & 0 & 0 &  0 &\cdots \\[-4pt]
                         \vdots & \ddots  & \ddots & 1 & 0   &  0 & \cdots \\[2pt]
                        \beta_{\vec{n}_r,r} & \beta_{\vec{n}_r,r-1} & \cdots & \beta_{\vec{n}_r,0} & 1 &  0 & \cdots \\[2pt]
                        0 & \beta_{\vec{n}_{r+1},r} & \beta_{\vec{n}_{r+1},r-1} & \cdots & \beta_{\vec{n}_{r+1},0} & 1   & \cdots \\[2pt]
                        0 & 0 & \beta_{\vec{n}_{r+2},r} & \beta_{\vec{n}_{r+2},r-1} & \cdots &  \beta_{\vec{n}_{r+2},0} & \cdots \\[2pt] 
                        0 & 0 & 0 & \ddots & \ddots & \ddots  & \ddots   
            \end{pmatrix}
       \begin{pmatrix} P_{\vec{n}_0}(x) \\ P_{\vec{n}_1}(x) \\ P_{\vec{n}_2}(x) \\ \vdots  \\ P_{\vec{n}_k}(x) \\ \vdots \end{pmatrix}. \]

\subsection{Christoffel-Darboux formula}
The Christoffel-Darboux kernel, which is the important reproducing kernel for orthogonal polynomials, has a counterpart in the theory of multiple orthogonal
polynomials. It uses both the type I and type II multiple orthogonal polynomials, and is a sum over a path from $\vec{0}$ to $\vec{n}$ as described before.
The Christoffel-Darboux kernel is defined as
\[     K_{\vec{n}}(x,y) = \sum_{k=0}^{N-1} P_{\vec{n}_k}(x)Q_{\vec{n}_{k+1}}(y)   \]
where $\vec{n}_0=\vec{0}$, $\vec{n}_N=\vec{n}$ and the path in $\mathbb{N}^r$ is such that $\vec{n}_{k+1}-\vec{n}_k=\vec{e}_i$ for some $i$ satisfying $1 \leq i \leq r$,
i.e., in every step the multi-index is increased by 1 in one component. This definition seems to depend on the choice of the path from $\vec{0}$ to $\vec{n}$, but 
surprisingly this kernel is independent of that chosen path. This is a consequence of the relations between the recurrence coefficients given by Theorem \ref{thm:multcomp}
and is best explained by the following analogue of the Christoffel-Darboux formula for orthogonal polynomials:
 
\begin{theorem}[Daems and Kuijlaars]
Let $(\vec{n}_k)_{0 \leq k \leq N}$ be a path in $\mathbb{N}^r$ starting from $\vec{n}_0=\vec{0}$ and ending in
$\vec{n}_N=\vec{n}$ (where $N = |\vec{n}|$), 
such that $\vec{n}_{k+1} - \vec{n}_k = \vec{e}_i$ for some $1 \leq i \leq r$. Then
\[  (x-y) \sum_{k=0}^{N-1}  P_{\vec{n}_k}(x)Q_{\vec{n}_{k+1}}(y) = P_{\vec{n}}(x)Q_{\vec{n}}(y)
    - \sum_{j=1}^r a_{\vec{n},j} P_{\vec{n}-\vec{e}_j}(x)Q_{\vec{n}+\vec{e}_j}(y) .  \]
\end{theorem}

\begin{proof}
This was first proved in \cite{DaemsKuijl2004} and a proof based on the nearest neighbor recurrence relations can be found in \cite{WVA2}.
\end{proof}

The sum depends only on the endpoint $\vec{n}$ of the path in $\mathbb{N}^r$ and not on the path from $\vec{0}$ to this point. In many cases this Christoffel-Darboux kernel
can be used to generate a determinantal process by using Theorem \ref{thm:detproc} and the biorthogonality in Property \ref{prop:biorth}. The only thing which is not
obvious is the positivity $K_{\vec{n}}(x,x) \geq 0$, which needs to be checked separately. See \cite{Kuijl} for more details about such determinantal processes.

\subsection{Hermite-Pad\'e approximation}
Multiple orthogonal polynomials have their roots in Hermite-Pad\'e approximation, which was introduced by Hermite and investigated in detail by Pad\'e (for $r=1$).
Hermite-Pad\'e approximation is a method to approximate $r$ functions simultaneously by rational functions. Multiple orthogonal polynomials appear when one
uses Hermite-Pad\'e approximation near infinity.
Let $(f_1,\ldots,f_r)$ be $r$ Markov functions, i.e.,
\[      f_j(z) = \int \frac{d\mu_j(x)}{z-x} = \sum_{k=0}^\infty \frac{m_k^{(j)}}{z^{k+1}} .  \]

\begin{definition}[Type I Hermite-Pad\'e approximation]
Type I Hermite-Pad\'e approximation is to find $r$ polynomials $(A_{\vec{n},1},\ldots,A_{\vec{n},r})$, with
$\deg A_{\vec{n},j} \leq n_j-1$, and a polynomial $B_{\vec{n}}$ such that
\begin{equation}  \label{HPI}
    \sum_{j=1}^r A_{\vec{n},j}(z) f_j(z) - B_{\vec{n}}(z) = \mathcal{O}\left( \frac{1}{z^{|\vec{n}|}} \right),
 \qquad z \to \infty. 
\end{equation}
\end{definition}

The solution is that $(A_{\vec{n},1},\ldots,A_{\vec{n},r})$ is the type I multiple orthogonal polynomial vector, and
\[    B_{\vec{n}}(z) = \int \sum_{j=1}^r \frac{A_{\vec{n},j}(z) - A_{\vec{n},j}(x)}{z-x} \, d\mu_j(x).  \]
The error in this approximation problem can also be expressed in terms of the type I multiple orthogonal polynomials. One has
\[ \sum_{j=1}^r A_{\vec{n},j}(z) f_j(z) - B_{\vec{n}}(z) = \int  \sum_{j=1}^r \frac{A_{\vec{n},j}(x)}{z-x} \, d\mu_j(x),  \]
and the orthogonality properties of the type I multiple orthogonal polynomials indeed show that \eqref{HPI} holds.

\begin{definition}[Type II Hermite-Pad\'e approximation]
Type II Hermite-Pad\'e approximation is to find a polynomial $P_{\vec{n}}$ of degree $\leq |\vec{n}|$
and polynomials $Q_{\vec{n},1}, \ldots,$ $Q_{\vec{n},r}$ such that
\begin{equation}   \label{HPII}
    P_{\vec{n}}(z) f_j(z) - Q_{\vec{n},j}(z) = \mathcal{O}\left( \frac{1}{z^{n_j+1}} \right),
 \qquad z \to \infty,  
\end{equation}
for $1 \leq j \leq r$.
\end{definition}

The solution for this approximation problem is to take the type II multiple orthogonal polynomial $P_{\vec{n}}$ and
\[    Q_{\vec{n},j}(z) = \int  \frac{P_{\vec{n}}(z) - P_{\vec{n}}(x)}{z-x} \, d\mu_j(x).  \]
Observe that this approximation problem is to find rational approximants to each $f_j$ with a \textit{common denominator}, and this
common denominator turns out to be the type II multiple orthogonal polynomial. 
The error can again be expressed in terms of the multiple orthogonal polynomial:
\[   P_{\vec{n}}(z) f_j(z) - Q_{\vec{n},j}(z) = \int \frac{P_{\vec{n}}(x)}{z-x}\, d\mu_j(x), \]
which can be verified by using the orthogonality conditions for the type II multiple orthogonal polynomial.

Hermite-Pad\'e approximants are used frequently in number theory to find good rational approximants for real numbers and to prove irrationality
and transcendence of some important real numbers. Hermite used these approximants (but at $0$ rather than $\infty$) to prove that $e$ is a transcendental number.
%In Section \ref{number} we will give some examples how multiple orthogonal polynomials allow irrationality proofs and bounds for the measure of
%irrationality of certain real numbers.

\subsection{Multiple Hermite polynomials}
As an example we will describe multiple Hermite polynomials in some detail and explain some applications where they are used.
The type II multiple Hermite polynomials $H_{\vec{n}}$ satisfy
\[   \int_{-\infty}^\infty H_{\vec{n}}(x) x^k e^{-x^2+c_jx}\, dx = 0, \qquad 0 \leq k \leq n_j-1 \]
for $1 \leq j \leq r$, with $c_i \neq c_j$ whenever $i \neq j$. This condition on the parameters $c_1,\ldots,c_r$ guarantees that every 
multi-index $\vec{n}$ is normal, since the measures with weight function $e^{-x^2+c_jx}$ $(1\leq j \leq r)$ form an AT-system. 
These multiple orthogonal polynomials can be obtained by using the \textit{Rodrigues formula}
\[    e^{-x^2} H_{\vec{n}}(x) = \frac{(-1)^{|\vec{n}|}}{2^{|\vec{n}|}} \left( \prod_{j=1}^r e^{-c_jx} \frac{d^{n_j}}{dx^{n_j}} e^{c_jx} \right) e^{-x^2}. \]

\noindent\shadowbox{\parbox{12cm}{
\begin{exercise}
Show that the differential operators
\[        e^{-c_jx} \frac{d^{n_j}}{dx^{n_j}} e^{c_j x}, \qquad  1 \leq j \leq r \]
are commuting. Use this (and integration by parts) to show that this indeed gives the type II multiple Hermite polynomial.
\end{exercise}}}

By using this Rodrigues formula (and the Leibniz rule for the $n$th derivative of a product), one finds the \textit{explicit expression}
\[  H_{\vec{n}}(x) = \frac{(-1)^{|\vec{n}|}}{2^{|\vec{n}|}} \sum_{k_1=0}^{n_1} \cdots \sum_{k_r=0}^{n_r} 
   \binom{n_1}{k_1} \cdots \binom{n_r}{k_r} c_1^{n_1-k_1} \cdots c_r^{n_r-k_r} (-1)^{|\vec{k}|} H_{|\vec{k}|}(x), \]
where $H_n$ are the usual Hermite polynomials.
The nearest neighbor recurrence relations for multiple Hermite polynomials are quite simple:
\[   xH_{\vec{n}}(x) = H_{\vec{n}+\vec{e}_k}(x) + \frac{c_k}{2} H_{\vec{n}}(x) + \frac12 \sum_{j=1}^r n_j H_{\vec{n}-\vec{e}_j}(x) , \qquad 1 \leq k \leq r. \]
They also have some useful differential properties: there are $r$ \textit{raising operators}
\[   \left(  e^{-x^2+c_jx} H_{\vec{n}-\vec{e}_j}(x) \right)' = -2 e^{-x^2+c_jx} H_{\vec{n}}(x), \qquad 1 \leq j \leq r,  \]
and one  \textit{lowering operator}
\[  H_{\vec{n}}'(x) = \sum_{j=1}^r n_j H_{\vec{n}-\vec{e}_j}(x).  \]
By combining these raising operators and the lowering operator one finds a \textit{differential equation} of order $r+1$: 
\[  \left( \prod_{j=1}^r D_j \right) D H_{\vec{n}}(x) = -2 \left( \sum_{j=1}^r n_j \prod_{i \neq j} D_i \right) H_{\vec{n}}(x), \]
where
\[    D = \frac{d}{dx}, \qquad  D_j = e^{x^2-c_jx} D e^{-x^2+c_jx}.  \]
One can also find some integral representations (see \cite{BleKuijl})
\[  H_{\vec{n}}(x) = \frac{1}{\sqrt{\pi} i} \int_{-i\infty}^{i\infty} e^{(s-x)^2} \prod_{j=1}^r \left( s- \frac{c_j}{2} \right)^{n_j}\, ds.  \]
For the type I multiple Hermite polynomials one has
\[  e^{-x^2+c_kx} A_{\vec{n},k}(x)= \frac{1}{\sqrt{\pi} 2\pi i} \oint_{\Gamma_k} e^{-(t-x)^2} \prod_{j=1}^r \left( t- \frac{c_j}2 \right)^{-n_j}\, dt , \]
where $\Gamma_k$ is a closed contour encircling $c_k/2$ once and none of the other $c_j/2$, and
\[  Q_{\vec{n}}(x) = \sum_{k=1}^r e^{-x^2+c_kx} A_{\vec{n},k}(x) = \frac{1}{\sqrt{\pi} 2\pi i} \oint_{\Gamma} e^{-(t-x)^2} \prod_{j=1}^r \left( t- \frac{c_j}2 \right)^{-n_j}\, dt, \]
where $\Gamma$ is a closed contour encircling all $c_j/2$.

\subsubsection{Random matrices}
These multiple Hermite polynomials are useful for investigating \textit{random matrices with external source} \cite{BleKuijlI}.
Let $\mathbf{M}$ be a random $N \times N$ Hermitian matrix and consider the \textit{ensemble} with probability
distribution
\[     \frac{1}{Z_N} \exp \Bigl( -\textup{Tr} (M^2 - AM) \Bigr)\ dM, \qquad  dM = \prod_{i=1}^N dM_{i,i} \prod_{1\leq i<j \leq N} dM_{i,j}, \]
where $A$ is a fixed  Hermitian matrix (the \textit{external source}).
The average characteristic polynomial is a multiple Hermite polynomial:

\begin{property}
Suppose $A$ has eigenvalues $c_1,\ldots,c_r$ with multiplicities $n_1,\ldots,n_r$, then
\[   \mathbb{E}\Bigl(\det(\mathbf{M}-zI_N) \Bigr) = (-1)^{|\vec{n}|} H_{\vec{n}}(z).   \]
\end{property}

Furthermore, the eigenvalues form a determinantal process with the Christoffel-Darboux kernel for multiple Hermite polynomials:  

\begin{property}
The density of the eigenvalues is given by
\[  P_N(\lambda_1,\ldots,\lambda_N) = \frac{1}{N!} \det  \Bigl( K_N(\lambda_i,\lambda_j) \Bigr)_{i,j=1}^{N}, \]
where the kernel is given by
\[   K_N(x,y) = e^{-(x^2+y^2)/2} \sum_{k=0}^{N-1} H_{\vec{n}_k}(x)Q_{\vec{n}_{k+1}}(y), \]
with $(\vec{n}_k)_{0\leq k \leq N}$ a path from $\vec{0}$ to $\vec{n}$ in $\mathbb{N}^r$ and
\[    Q_{\vec{n}}(y) = \sum_{j=1}^r A_{\vec{n},j}(y) e^{c_jy} . \]
\end{property}

\noindent This means that we can also find the correlation functions:

\begin{property}
The $m$-point correlation function 
%{\small
\[  \rho_m(\lambda_1,\ldots,\lambda_m) = \frac{N!}{(N-m)!} \int_{-\infty}^\infty \cdots \int_{-\infty}^\infty P_N(\lambda_1,\ldots,\lambda_N)
  \, d\lambda_{m+1} \ldots d\lambda_N \]
is given by 
\[  \rho_m(\lambda_1,\ldots,\lambda_m) =  \det  \Bigl( K_N(\lambda_i,\lambda_j) \Bigr)_{i,j=1}^{m}, \]
where the kernel is given by
\[   K_N(x,y) = e^{-(x^2+y^2)/2} \sum_{k=0}^{N-1} H_{\vec{n}_k}(x)Q_{\vec{n}_{k+1}}(y). \]
\end{property}

\subsubsection{Non-intersecting Brownian motions}
Another interesting problem where multiple Hermite polynomials are appearing is to find what happens with $n$ independent Brownian motions
(in fact, $n$ Brownian bridges)
with the constraint that they are not allowed to intersect, see \cite{DaemsKuijl}.

\begin{figure}[tbp] % float placement: (h)ere, page (t)op, page (b)ottom, other (p)age
  \centering
  \includegraphics[width=4in]{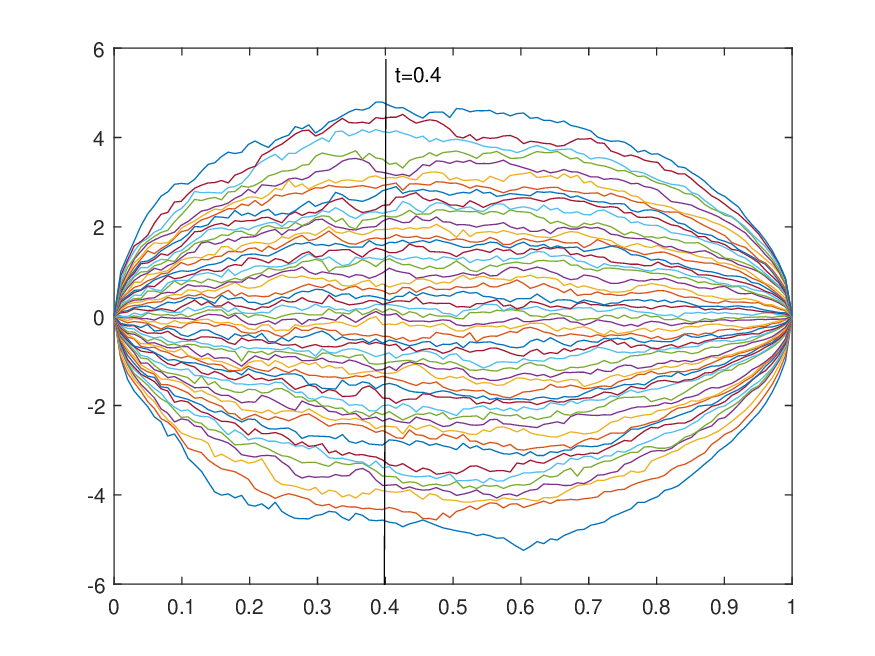}
  \caption{Non-intersecting Brownian motions}
  \label{fig:Brown}
\end{figure}

The density of the probability that the $n$ non-intersecting paths, leaving $(t=0)$ at $a_1,\ldots,a_n$ 
and arriving $(t=1)$ at $b_1,\ldots,b_n$, are at $x_1,\ldots,x_n$ at time $t \in (0,1)$ 
is (Karlin and McGregor \cite{KM}) 
\[    p_{n,t}(x_1,\ldots,x_n) = \frac{1}{Z_n} \det \Bigl( P(t,a_j,x_k) \Bigr)_{j,k=1}^n
                                              \det \Bigl( P(1-t,b_j,x_k) \Bigr)_{j,k=1}^n  , \]
where
\[       P(t,a,x) = \frac{1}{\sqrt{2\pi t}} e^{-\frac1{2t} (x-a)^2}.  \]
When $a_1, \ldots,a_n \to 0$ and $b_1,\ldots,b_n \to 0$ (see Fig. \ref{fig:Brown}) then
\[    p_{n,t}(x_1,\ldots,x_n) = \frac{1}{n!} \det \Bigl( K_n(x_j,x_k) \Bigr)_{j,k=1}^n , \]
where the kernel is given by
\[    K_n(x,y) = e^{-\frac{x^2}{4t} - \frac{y^2}{4(1-t)}} \sum_{k=0}^{n-1} H_k(\frac{x}{\sqrt{2t}})H_k(\frac{y}{\sqrt{2(1-t)}}).   \]
This kernel is related to the Christoffel-Darboux kernel for the usual Hermite polynomials.

When $a_1, \ldots,a_n \to 0$ and $b_1,\ldots,b_{n/2} \to -b$, $b_{n/2+1},\ldots,b_n \to b$ (see Fig. \ref{fig:Brown2}) then
\[    p_{n,t}(x_1,\ldots,x_n) = \frac{1}{n!} \det \Bigl( K_n(x_j,x_k) \Bigr)_{j,k=1}^n , \]
with 
\[    K_n(x,y) = e^{-\frac{x^2}{4t} - \frac{y^2}{4(1-t)}} \sum_{k=0}^{n-1} H_{\vec{n}_k} (\frac{x}{\sqrt{2t}}) 
Q_{\vec{n}_{k+1}} (\frac{y}{\sqrt{2(1-t)}}) ,  \]
with multiple orthogonal polynomials for the weights 
\[    e^{-x^2-2bx}, \quad e^{-x^2+2bx}. \]
This kernel is related to the Christoffel-Darboux kernel for multiple Hermite polynomials.
An interesting phenomenon appears:
for small values of $t$ the points at level $t$ accumulate on one interval, but for larger values of $t$ in $(0,1)$ the points 
accumulate on two disjoint intervals. There is a phase transition at a critical point $t_c \in (0,1)$. A detailed asymptotic analysis
of the kernel near this point will require a special function satisfying a third order differential equation (the Pearsey equation) which is a limiting case of the
third order differential equation of multiple Hermite polynomials. The limiting kernel is known as the Pearsey kernel.

\begin{figure}[tbp] % float placement: (h)ere, page (t)op, page (b)ottom, other (p)age
  \centering
  \includegraphics[width=3.2in]{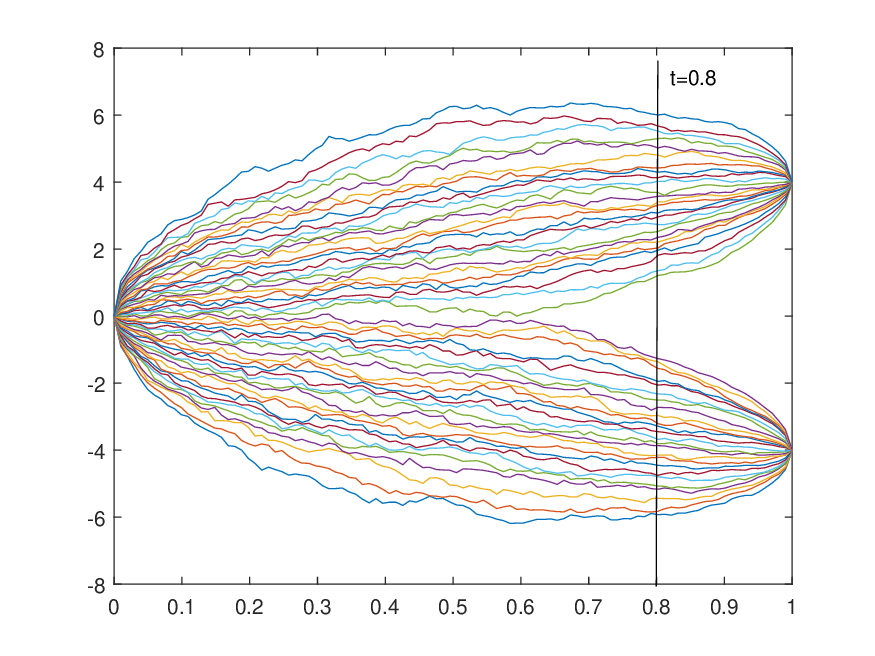}
  \caption{Non-intersecting Brownian motions (two arriving points)}
  \label{fig:Brown2}
\end{figure}

%\subsubsection{Non-intersecting squared Bessel paths}
%\begin{figure}[tbp] % float placement: (h)ere, page (t)op, page (b)ottom, other (p)age
%  \centering
%  \includegraphics[width=3.2in]{newBESQ2.pdf}
%  \caption{Non-intersecting squared Bessel paths}
%  \label{fig:Bessel}
%\end{figure}

\subsection{Multiple Laguerre polynomials}

The Laguerre weight is 
\[   w(x) = x^\alpha e^{-x}, \qquad x \in [0,\infty), \ \alpha > -1.  \]
There are two easy ways to obtain multiple Laguerre polynomials:
\begin{enumerate}
  \item Changing the parameter $\alpha$ to $\alpha_1,\ldots,\alpha_r$. 
    This gives \textit{multiple Laguerre polynomials of the first kind}.
  \item Changing the exponential decay at infinity from $e^{-x}$ to $e^{-c_jx}$ with parameters $c_1,\ldots,c_r$.
    This gives \textit{multiple Laguerre polynomials of the second kind}.
\end{enumerate}   

\subsubsection{Multiple Laguerre polynomials of the first kind}
Type II multiple Laguerre of the first kind $L_{\vec{n}}^{\vec{\alpha}}(x)$ satisfy
\[   \int_0^\infty x^k L_{\vec{n}}^{\vec{\alpha}}(x) x^{\alpha_j} e^{-x}\, dx = 0, \qquad 0 \leq k \leq n_j-1, \]
for $1 \leq j \leq r$. In order that all multi-indices are normal we need to have parameters $\alpha_j > -1$ and $\alpha_i - \alpha_j \notin \mathbb{Z}$ 
whenever $i \neq j$, in which case the $r$ measures form an AT-system.
The multiple orthogonal polynomials can be found from the \textit{Rodrigues formula}
\[    (-1)^{|\vec{n}|} e^{-x} L_{\vec{n}}^{\vec{\alpha}}(x) = \prod_{j=1}^r \left( x^{-\alpha_j} \frac{d^{n_j}}{dx^{n_j}} x^{n_j+\alpha_j} \right) e^{-x}.  \]
An explicit formula is 
\begin{multline*}
   L_{\vec{n}}^{\vec{\alpha}}(x) = \sum_{k_1=0}^{n_1} \cdots \sum_{k_r=0}^{n_r} (-1)^{|\vec{k}|} \frac{n_1!}{(n_1-k_1)!} \cdots
   \frac{n_r!}{(n_r-k_r)!} \\
  \times \binom{n_r+\alpha_r}{k_r} \binom{n_r+n_{r-1}+\alpha_{r-1}-k_r}{k_{r-1}} \cdots \binom{|\vec{n}|-|\vec{k}|+k_1+\alpha_1}{k_1} 
  x^{|\vec{n}|-|\vec{k}|}  .
\end{multline*}
Another explicit expression with hypergeometric functions is
\[  (-1)^{|\vec{n}|} e^{-x} L_{\vec{n}}^{\vec{\alpha}}(x) = \prod_{j=1}^r (\alpha_j+1)_{n_j} \ {}_rF_r\left(
    \left. \begin{array}{c} n_1+\alpha_1+1,\ldots,n_r+\alpha_r+1 \\ \alpha_1+1,\ldots,\alpha_r+1 \end{array} \right| -x \right) . \]
The nearest neighbor recurrence relations are
\[    xL_{\vec{n}}(x) = L_{\vec{n}+\vec{e}_k}(x) + b_{\vec{n},k} L_{\vec{n}}(x) + \sum_{j=1}^r a_{\vec{n},j} L_{\vec{n}-\vec{e}_j}(x) \]
with
\[     a_{\vec{n},j} = n_j (n_j+\alpha_j) \prod_{i=1,i\neq j}^{r} \frac{n_j+\alpha_j-\alpha_i}{n_j-n_i+\alpha_j-\alpha_i}, \]
and
\[     b_{\vec{n},k} = |\vec{n}|+n_k+\alpha_k+1.  \]
These multiple Laguerre polynomials also have some differential properties.
There are $r$ \textit{raising operators}
\[  \frac{d}{dx} \left( x^{\alpha_j+1} e^{-x}L_{\vec{n}-\vec{e}_j}^{\vec{\alpha}+\vec{e}_j}(x) \right) 
   = -x^{\alpha_j} e^{-x} L_{\vec{n}}^{\vec{\alpha}}(x),  \qquad 1 \leq j \leq r.  \]
and there is one \textit{lowering operator}
\[  \frac{d}{dx} L_{\vec{n}}^{\vec{\alpha}}(x) = \sum_{j=1}^r \frac{\prod_{i=1}^r (n_i+\alpha_i-\alpha_j)}{\prod_{i=1,i\neq j}^r (\alpha_i-\alpha_j)}
     L_{\vec{n}-\vec{e}_j}^{\vec{\alpha}+\vec{e}_j}(x).  \]
Combining them gives the \textit{differential equation}
\[  \left( \prod_{j=1}^r D_j\right) D L_{\vec{n}}^{\vec{\alpha}}(x) 
  = - \sum_{j=1}^r \frac{\prod_{i=1}^r (n_i+\alpha_i-\alpha_j)}{\prod_{i=1,i\neq j}^r (\alpha_i-\alpha_j)} 
 \left( \prod_{i \neq j} D_i \right) L_{\vec{n}}^{\vec{\alpha}}(x).  \]
\[     D = \frac{d}{dx}, \qquad    D_j = x^{-\alpha_j} e^{x} D x^{\alpha_j+1} e^{-x}.  \] 

\subsubsection{Multiple Laguerre polynomials of the second kind}
Type II multiple Laguerre polynomials of the second kind $L_{\vec{n}}^{\alpha,\vec{c}}(x)$ satisfy
\[  \int_0^\infty x^k L_{\vec{n}}^{\alpha,\vec{c}}(x) x^\alpha e^{-c_jx}\, dx = 0, \qquad 0 \leq k \leq n_j-1, \]
for $1 \leq j \leq r$. The parameters need to satisfy $\alpha > -1$ and $c_j >0$ with $c_i \neq c_j$ whenever $i \neq j$.
The \textit{Rodrigues formula} is
\[    (-1)^{|\vec{n}|} \prod_{j=1}^r c_j^{n_j} \ x^\alpha L_{\vec{n}}^{\alpha,\vec{c}}(x)
    = \prod_{j=1}^r \left( e^{c_jx} \frac{d^{n_j}}{dx^{n_j}} e^{-c_j x} \right)  \ x^{|\vec{n}|+\alpha}, \]
which allows to find the explicit expression
\[ L_{\vec{n}}^{\alpha,\vec{c}}(x) = \sum_{k_1=0}^{n_1} \cdots \sum_{k_r=0}^{n_r} \binom{n_1}{k_1} \cdots \binom{n_r}{k_r} 
     \binom{|\vec{n}|+\alpha}{|\vec{k}|} (-1)^{|\vec{k}|}
     \frac{|\vec{k}|!}{c_1^{k_1} \cdots c_r^{k_r}} x^{|\vec{n}|-|\vec{k}|} . \]
The nearest neighbor recurrence relations are
\[    xL_{\vec{n}}(x) = L_{\vec{n}+\vec{e}_k}(x) + b_{\vec{n},k} L_{\vec{n}}(x) + \sum_{j=1}^r a_{\vec{n},j} L_{\vec{n}-\vec{e}_j}(x), \]
with
\[     a_{\vec{n},j} = \frac{n_j(|\vec{n}|+\alpha)}{c_j^2}, \quad  b_{\vec{n},k} = \frac{|\vec{n}|+\alpha+1}{c_k} + \sum_{j=1}^r \frac{n_j}{c_j}. \] 
The differential properties include $r$ \textit{raising operators}
\[    \frac{d}{dx} \left( x^{\alpha+1} e^{-c_jx} L_{\vec{n}-\vec{e}_j}^{\alpha+1,\vec{c}}(x)\right) = - c_j x^{\alpha} e^{-c_jx} L_{\vec{n}}^{\alpha,\vec{c}}(x),
   \qquad 1 \leq j \leq r. \]
and one \textit{lowering operator}
\[  \frac{d}{dx} L_{\vec{n}}^{\alpha,\vec{c}}(x) = \sum_{j=1}^r n_j L_{\vec{n}-\vec{e}_j}^{\alpha+1,\vec{c}}(x).  \]
They give the \textit{differential equation}
\[   \left( \prod_{j=1}^rD_j \right) x^{\alpha+1} D L_{\vec{n}}^{\alpha,\vec{c}}(x) = - \sum_{j=1}^r c_j n_j \left( \prod_{i\neq j} D_i \right) x^\alpha
     L_{\vec{n}}^{\alpha,\vec{c}}(x), \]
where
\[   D = \frac{d}{dx}, \quad D_j = e^{c_jx} D e^{-c_jx} .  \]

\subsubsection{Random matrices: Wishart ensemble}
Wishart (1928) introduced the \textit{Wishart distribution} for $N\times N$ positive definite Hermitian matrices
\[ \mathbf{M} = \mathbf{X} \mathbf{X}^*, \qquad \mathbf{X} \in \mathbb{C}^{N\times (N+p)} ,  \]
where all the columns of $\mathbf{X}$ are independent and
have a multivariate Gauss distribution with covariance matrix $\Sigma$. The density for the Wishart distribution is
\[    \frac{1}{Z_N} e^{-\textup{Tr}(\Sigma^{-1} M)} (\det M)^p\, dM.   \]
If $\Sigma = I_N$ then Laguerre polynomials (with $\alpha = p$) play an important role.
If $\Sigma^{-1}$ has eigenvalues $c_1,\ldots,c_r$ with multiplicities $n_1,\ldots,n_r$, then we need multiple
Laguerre polynomials of the second kind. The average characteristic polynomial is
\[    \mathbb{E} \Bigl( \det (\mathbf{M}-zI_N)\Bigr) = (-1)^{|\vec{n}|} L_{\vec{n}}^{p,\vec{c}}(z).  \]

\subsection{Jacobi-Pi\~neiro polynomials}   
There are several ways to find multiple Jacobi polynomials. Here we only mention one way which uses the same differential operators as the multiple Laguerre polynomials
of the first kind. The Jacobi-Pi\~neiro polynomials $P_{\vec{n}}^{(\vec{\alpha},\beta)}$ satisfy
\[   \int_0^1 P_{\vec{n}}^{(\vec{\alpha},\beta)}(x) x^k x^{\alpha_j} (1-x)^\beta \, dx = 0, \qquad 0 \leq k \leq n_j-1, \]
for $1 \leq j \leq r$. Hence we are using Jacobi weights $x^\alpha (1-x)^\beta$ on the interval $[0,1]$, with $\alpha,\beta >-1$ but with $r$ different parameters $\alpha_1,\ldots,\alpha_r$.
In order to have a perfect system we require $\alpha_i-\alpha_j \notin \mathbb{Z}$ whenever $i \neq j$. They can be obtained using the \textit{Rodrigues formula}
\begin{multline*}  
   (-1)^{|\vec{n}|} \left( \prod_{j=1}^r (|\vec{n}|+\alpha_j+\beta)_{n_j} \right) (1-x)^\beta P_{\vec{n}}^{(\vec{\alpha},\beta)}(x) \\
     = \prod_{j=1}^r \left( x^{-\alpha_j} \frac{d^{n_j}}{dx^{n_j}} x^{n_j+\alpha_j} \right) (1-x)^{\beta+|\vec{n}|}.  
\end{multline*}
An expression in terms of generalized hypergeometric functions is
\begin{multline*}
    (-1)^{|\vec{n}|} \left( \prod_{j=1}^r (|\vec{n}|+\alpha_j+\beta)_{n_j} \right) (1-x)^\beta P_{\vec{n}}^{(\vec{\alpha},\beta)}(x)  \\
    = \prod_{j=1}^r (\alpha_j+1)_{n_j} \  \left. {}_{r+1}F_r \left(  \begin{array}{c} -|\vec{n}|-\beta,\alpha_1+n_1+1,\ldots,\alpha_r+n_r+1 \\
                                                                       \alpha_1+1,\ldots,\alpha_r+1   \end{array} \right|  x \right).
\end{multline*}
This hypergeometric function does not terminate when $\beta$ is not an integer. Another useful expression is
\begin{multline*}
   (-1)^{|\vec{n}|}   P_{\vec{n}}^{(\vec{\alpha},\beta)}(x)  \\ 
   = \frac{n_1!\cdots n_r!}{\prod_{j=1}^r (|\vec{n}|+\alpha_j+\beta)_{n_j}} \sum_{k_1=0}^{n_1} \cdots \sum_{k_r=0}^{n_r} (-1)^{|\vec{k}|} \prod_{j=1}^r \binom{n_j+\alpha_j+ \sum_{i=1}^{j-1} k_i}{n_j-k_j} \\
             \times      \binom{|\vec{n}|+\beta}{|\vec{k}|} \frac{|\vec{k}|! x^{|\vec{k}|} (1-x)^{|\vec{n}|-|\vec{k}|}}{k_1!\cdots k_r!} .
\end{multline*}
Again there are $r$ raising differential operators and one lowering operator and the recurrence coefficients are known explicitly.
These polynomials are useful for rational approximation of polylogarithms, and in particular for the zeta function $\zeta(k)$ at integers.
The polylogarithms are defined by
\[    \textup{Li}_k(z) = \sum_{n=1}^\infty \frac{z^n}{n^k}, \qquad |z| < 1, \]
and one has
\[    \textup{Li}_{k+1}(1/z) = \frac{(-1)^k}{k!} \int_0^1 \frac{\log^k(x)}{z-x}\, dx .  \]
Simultaneous rational approximation to $\textup{Li}_1(1/z),\ldots,\textup{Li}_r(1/z)$ can be done using Hermite-Pad\'e approximation with a limiting case of Jacobi-Pi\~neiro
polynomials where $\beta=0$ and $\alpha_1=\alpha_2=\cdots=\alpha_r=0$, which is possible when $n_1\geq n_2\geq \cdots \geq n_r$.
This is particularly interesting if we let $z \to 1$, since $\textup{Li}_k(1) = \zeta(k)$. Ap\'ery's construction of good rational approximants for $\zeta(3)$
(proving that $\zeta(3)$ is irrational) essentially makes use of these multiple orthogonal polynomials, see, e.g. \cite{WVA4}.

\section{Orthogonal polynomials and Painlev\'e equations}
In this section we describe how orthogonal polynomials are related to non-linear difference and differential equations,
in particular to discrete Painlev\'e equations and the six Painlev\'e differential equations. For a recent discussion
on this relation between orthogonal polynomials and Painlev\'e equations we refer to the monograph \cite{WVA3}. Other useful references are \cite{Clarkson2,Clarkson,WVA1}.

Painlev\'e equations (discrete and continuous) appear at various places in the theory of orthogonal polynomials, in particular
\begin{itemize}
   \item The recurrence coefficients of some semiclassical orthogonal polynomials satisfy discrete Painlev\'e equations. 
   \item The recurrence coefficients of orthogonal polynomials with a Toda-type evolution satisfy Painlev\'e differential equations for which special
         solutions depending on special functions (Airy, Bessel, (confluent) hypergeometric, parabolic cylinder functions) are relevant.  
   \item Rational solutions of Painlev\'e equations can be expressed in terms of Wronskians of orthogonal polynomials.   
   \item The local asymptotics for orthogonal polynomials at critical points is often using special transcendental solutions of Painlev\'e equations.
   \end{itemize}
In this section we will only deal with the first two of these.

What are Painlev\'e (differential) equations?
They are second order nonlinear differential equations 
  \[ y'' = R(y',y,x), \qquad  R \textup{ rational}, \] 
that have the \textit{Painlev\'e property}: \textbf{The general solution is free from movable branch points}. The only singularities which may depend on the initial conditions are poles. 
Painlev\'e and his collaborators found 50 families (up to M\"obius transformations), all of which could be reduced to known equations and six new equations (new at least at the beginning of the 20th century). 
The six Painlev\'e equations are 
\begin{eqnarray} 
  \textup{P}_{\scriptstyle\textup{I}} & & y'' = 6y^2 + x, \nonumber \\
  \textup{P}_{\scriptstyle\textup{II}} & & y'' = 2y^3 + xy + \alpha, \label{P2} \\
  \textup{P}_{\scriptstyle\textup{III}} & & y'' = \frac{(y')^2}{y} - \frac{y'}{x} + \frac{\alpha y^2+\beta}{x} + \gamma y^3 + \frac{\delta}{y},  \label{P3} \\
  \textup{P}_{\scriptstyle\textup{IV}} & & y'' = \frac{(y')^2}{2y} + \frac32 y^3 + 4xy^2 + 2(x^2-\alpha)y + \frac{\beta}{y}, \label{P4} \\
  \textup{P}_{\scriptstyle\textup{V}} & & y'' = \left( \frac{1}{2y} + \frac{1}{y-1} \right) (y')^2 - \frac{y'}{x} + \frac{(y-1)^2}{x^2} \left( \alpha y + 
  \frac{\beta}{y} \right) + \frac{\gamma y}{x} \nonumber \\
   & & \qquad +\ \frac{\delta y(y+1)}{y-1}, \label{P5} \qquad\\
  \textup{P}_{\scriptstyle\textup{VI}} & & y'' = \frac12 \left( \frac{1}{y} + \frac{1}{y-1} + \frac{1}{y-x} \right) (y')^2 
     - \left(\frac{1}{x} + \frac{1}{x-1} + \frac{1}{y-x} \right) y'  \nonumber \\ 
      & & \qquad +\ \frac{y(y-1)(y-x)}{x^2(x-1)^2} 
     \left(\alpha + \frac{\beta x}{y^2} + \frac{\gamma (x-1)}{(y-1)^2} + \frac{\delta x(x-1)}{(y-x)^2} \right),  \nonumber
\end{eqnarray}

Discrete Painlev\'e equations are somewhat more difficult to describe. Roughly speaking they are second order nonlinear recurrence equations for which the continuous limit is a Painlev\'e equation. 
They have the \textit{singularity confinement} property,  but this property is not sufficient to characterize discrete Painlev\'e equations.
A quote by Kruskal \cite{Kruskal} is:
  \begin{quote} Anything simpler becomes trivially integrable, anything more complicated becomes hopelessly non-integrable.
  \end{quote} 
A more correct description is that they are nonlinear recurrence relations with `nice' symmetry and geometry.
A full classification of discrete (and continuous) Painlev\'e equations has been found by Sakai \cite{Sakai}. This is based on \textit{rational surfaces} 
associated with affine root systems. It describes the space of initial values which parametrizes all the solutions (Okamoto \cite{Okamoto79}).
A fine tuning of this classification was given recently by Kajiwara, Noumi and Yamada \cite{KNY}: they also include the \textit{symmetry}, i.e., 
the group of B\"acklund transformations,
which are transformations that map a solution of a Painlev\'e equation to another solution with different parameters.
A partial list of discrete Painlev\'e equations is:
\begin{eqnarray}   
  \textup{d-P}_{\scriptstyle\textup{I}} & & x_{n+1}+x_n + x_{n-1} = \frac{z_n + a(-1)^n}{x_n} + b,  \label{d-P1} \\
  \textup{d-P}_{\scriptstyle\textup{II}} & & x_{n+1} + x_{n-1} = \frac{x_n z_n + a}{1-x_n^2},  \label{d-P2} \\
  \textup{d-P}_{\scriptstyle\textup{IV}} & & (x_{n+1}+x_n)(x_n+x_{n-1}) = \frac{(x_n^2-a^2)(x_n^2-b^2)}{(x_n+z_n)^2 - c^2}, \nonumber \\
  \textup{d-P}_{\scriptstyle\textup{V}} & & \frac{(x_{n+1}+x_n-z_{n+1}-z_n)(x_n+x_{n-1}-z_n-z_{n-1})}{(x_{n+1}+x_n)(x_n+x_{n-1})} \nonumber\\
  &&   \qquad = \frac{[(x_n-z_n)^2-a^2][(x_n-z_n)^2-b^2]}{(x_n-c^2)(x_n-d^2)},   \nonumber
\end{eqnarray}  
where $z_n = \alpha n + \beta$ and $a,b,c,d$ are constants. 
%\hskip-.6cm\parbox{4in}{
\begin{eqnarray*}
%\textup{q-P}_{\scriptstyle\textup{II}} & & (x_{n+1}x_n-1)(x_nx_{n-1}-1) = \frac{q_nq_{n-1} x_n}{a^2(x_n-aq_n)}, \\
\textup{q-P}_{\scriptstyle\textup{III}} & &  x_{n+1}x_{n-1} = \frac{(x_n-aq_n)(x_n-bq_n)}{(1-cx_n)(1-x_n/c)}, \\
\textup{q-P}_{\scriptstyle\textup{V}} & &   (x_{n+1}x_n-1)(x_nx_{n-1}-1) = 
        \frac{(x_n-a)(x_n-1/a)(x_n-b)(x_n-1/b)}{(1-cx_nq_n)(1-x_nq_n/c)}, \qquad \quad \\
\textup{q-P}_{\scriptstyle\textup{VI}} & & \frac{(x_nx_{n+1}-q_nq_{n+1})(x_nx_{n-1}-q_nq_{n-1})}{(x_nx_{n+1}-1)(x_nx_{n-1}-1)} \nonumber \\
          & & \qquad = \frac{(x_n-aq_n)(x_n-q_n/a)(x_n-bq_n)(x_n-q_n/b)}{(x_n-c)(x_n-1/c)(x_n-d)(x_n-1/d)},
\end{eqnarray*}
where $q_n = q_0 q^n$ and $a,b,c,d$ are constants. 
\begin{eqnarray*}    \label{adPIV}
  \alpha\textup{-d-P}_{\scriptstyle\textup{IV}} & & (x_n+y_n)(x_{n+1}+y_n) = 
 \frac{(y_n-a)(y_n-b)(y_n-c)(y_n-d)}{(y_n+\gamma-z_n)(y_n-\gamma-z_n)} \nonumber \\
  & & (x_n+y_n)(x_n+y_{n-1}) = \frac{(x_n+a)(x_n+b)(x_n+c)(x_n+d)}{(x_n+\delta-z_{n+1/2})(x_n-\delta-z_{n+1/2})} .
\end{eqnarray*}
The latter corresponds to $\textup{d-P}(E_6^{(1)}/A_2^{(1)})$ where $E_6^{(1)}$ is the surface type and $A_2^{(1)}$ is the symmetry type.
Sakai's classification (surface type) corresponds to the following diagram:

{\small
\setlength{\tabcolsep}{2pt}
\begin{tabular}{ccccccccccccccc} 
$E_8^e$  & & & & & & & & & & & & &  $A_1^q$  & \\
$\downarrow$ & & & & & & & & & & & & & \hskip-15pt $\nearrow$ & \\
$E_8^q$ & $\rightarrow$ & $E_7^q$ & $\rightarrow$ & $E_6^q$ & $\rightarrow$ & $D_5^q$ & $\rightarrow$ & $A_4^q$ & $\rightarrow$ & $(A_2+A_1)^q$ &
  $\rightarrow$ & $(A_1+A_1)^q$ & $\rightarrow$ &  $A_1^q$ \\ 
$\downarrow$ & & $\downarrow$ & & $\downarrow$ & & $\downarrow$ & & $\downarrow$ & & $\quad\ \vert \qquad \downarrow \qquad $ & & $\vert \qquad \downarrow$ && \\
$E_8^d$ & $\rightarrow$ & $E_7^d$ & $\rightarrow$ & $E_6^d$ & $\rightarrow$ & $D_4^c$ & $\rightarrow$ & $A_3^c$ & $\rightarrow$ & $\vert \quad (2A_1)^c$ &
 $\rightarrow$ &  $\ \vert \qquad A_1^c$ & \\
& & & & & & & & & $\searrow$ & $\ \ \downarrow \qquad\qquad$ & $\searrow$ & $\quad\  \downarrow \qquad\qquad$ & & \\
& & & & & & & & & & $A_2^c$ \qquad\quad & $\rightarrow$ & $A_1^c$ \qquad \quad & & 
\end{tabular}}

\subsection{Compatibility and Lax pairs}
There is a general philosophy behind the reason why Painlev\'e equations appear for the recurrence coefficients of orthogonal polynomials.
Orthogonal polynomials $P_n(x)$ are really functions of two variables: a discrete variable $n$ and a continuous variable $x$. The three
term recurrence relation \eqref{M3TRR}
 gives a difference equation in the variable $n$, and if the measure is absolutely continuous with a weight function
$w$ that satisfies a Pearson equation
\begin{equation}  \label{Pearson}
    \frac{d}{dx} [\sigma(x) w(x)] = \tau(x) w(x), 
\end{equation}
where $\sigma$ and $\tau$ are polynomials, then the orthogonal polynomials also satisfy differential relations in the variable $x$.
If $\deg \sigma \leq 2$ and $\deg \tau = 1$ then we are dealing with classical orthogonal polynomials which satisfy the second order
differential equation
\[    \sigma(x) y''(x) + \tau(x) y'(x) + \lambda_n y(x) = 0 , \]
where $\lambda_n =  -n(n-1) \sigma''/2 - n \tau'$. In the semiclassical case we still have the Pearson equation \eqref{Pearson} but 
we allow $\deg \sigma > 2$ or $\deg \tau \neq 1$. In that case there is a \textit{structure relation}
\begin{equation}  \label{structure}  
   \sigma(x) \frac{d}{dx} P_n(x) = \sum_{k=n-t}^{n+s-1} A_{n,k} P_k(x) ,
\end{equation}
where $s= \deg \sigma$ and $t = \max \{ \deg \tau, \deg \sigma-1 \}$.
The structure relation \eqref{structure} and the three-term recurrence relation \eqref{M3TRR} have to be \textit{compatible}: if we
differentiate the terms in the recurrence relation \eqref{M3TRR} and replace all the $P_k'(x)$ using the structure relation \eqref{structure},
then we get a linear combination of a finite number of orthogonal polynomials that is equal to $0$. Since (orthogonal) polynomials are linearly
independent in the linear space of polynomials, the coefficients in this linear combination have to be zero, and this gives relations
between the recurrence coefficient $a_n^2, b_n$ and the coefficients $A_{n,k}$ in the structure relation. Eliminating these $A_{n,k}$ gives
recurrence relations for the $a_n^2,b_n$, which turn out to be non-linear. If they are of second order, then we can identify them as
discrete Painlev\'e equations. In this way the three-term recurrence relation and the structure relation can be considered as a \textit{Lax pair} for the
obtained discrete Painlev\'e equation. 

In order to get to the Painlev\'e differential equation, we need to introduce an extra continuous parameter $t$. For this we will 
use an exponential modification of the measure $\mu$ and investigate orthogonal polynomials for the measure $d\mu_t(x) = e^{xt} \, d\mu(x)$,
whenever all the moments of this modified measure exist. We will denote the monic orthogonal polynomials by $P_n(x;t)$ and in this way the
orthogonal polynomial is now a function of three variables $n,x,t$. The behavior for the parameter $t$ is given by:
\begin{theorem}  \label{thm:Toda}
The monic orthogonal polynomials $P_n(x;t)$ for the measure $d\mu_t(x) = e^{xt}\, d\mu(x)$ satisfy
\begin{equation}   \label{TodaAnsatz}
     \frac{d}{dt} P_n(x;t) = C_n(t) P_{n-1}(x;t),  
\end{equation}
where $C_n(t)$ depends only on $t$ and $n$.
\end{theorem}

\begin{proof}
First of all, since $P_n(x;t)$ is a monic polynomial, the derivative $\frac{d}{dt} P_n(x;t)$ is a polynomial of degree $\leq n-1$. We will show that
it is orthogonal to $x^k$ for $0 \leq k \leq n-2$ for the measure $e^{xt}\, d\mu(x)$, so that it is proportional to $P_{n-1}(x;t)$, which proves
\eqref{TodaAnsatz}. We start from the orthogonality relations
\[   \int P_n(x;t) x^k e^{xt} \, d\mu(x) = 0, \qquad  0 \leq k \leq n-1, \]
and take derivatives with respect to $t$ to find
\[    \int \left( \frac{d}{dt} P_n(x;t) \right) x^k e^{xt}\, d\mu(x) + \int P_n(x;t) x^{k+1} e^{xt}\, d\mu(x) = 0, \qquad 0 \leq k \leq n-1. \]
The second integral vanishes for $0 \leq k \leq n-2$ by orthogonality, hence
\[   \int \left( \frac{d}{dt} P_n(x;t) \right) x^k e^{xt}\, d\mu(x) = 0, \qquad 0 \leq k \leq n-2, \]
which is what we needed to prove.
\end{proof}

This relation is not new, see e.g. \cite[\S 4]{SpiZhed}, but has not been sufficiently appreciated in the literature.
If we now check the compatibility between \eqref{TodaAnsatz} and the three-term recurrence relation \eqref{M3TRR}, then we find differential-difference
equations for the recurrence coefficients $a_n^2, b_n$.

\begin{theorem}[Toda equations]  \label{thm:Toda}
The recurrence coefficients $a_n^2(t)$ and $b_n(t)$ for the orthogonal polynomials $P_n(x;t)$ satisfy
\begin{eqnarray}
    \frac{d}{dt} a_n^2(t) &=& a_n^2 (b_n-b_{n-1}), \qquad n \geq 1,  \label{Toda-a}\\
    \frac{d}{dt} b_n(t) &=&  a_{n+1}^2 -a_n^2, \qquad n \geq 0,  \label{Toda-b}
\end{eqnarray}
with $a_0^2=0$.
\end{theorem}

\begin{proof}
If we take derivatives with respect to $t$ in the three-term recurrence relations \eqref{M3TRR}, then
\begin{multline*} 
    x \frac{d}{dt} P_n(x;t) = \frac{d}{dt} P_n(x;t) + b_n'(t) P_n(x;t) + b_n \frac{d}{dt} P_n(x) \\ 
   +\ (a_n^2)'(t) P_{n-1}(x;t)  + a_n^2 \frac{d}{dt} P_{n-1}(x;t). 
\end{multline*}
Use \eqref{TodaAnsatz} to find
\begin{multline*}
   xC_n P_{n-1}(x;t) = C_{n+1} P_n(x;t) + b_n' P_n(x;t) + b_n C_n P_{n-1}(x;t) \\
  +\ (a_n^2)' P_{n-1}(x;t) + a_n^2 C_{n-1} P_{n-2}(x;t).  
\end{multline*}
If we compare this with \eqref{M3TRR} (with $n$ shifted to $n-1$), then we find
\begin{eqnarray} 
      C_{n+1} + b_n' &=& C_n,   \label{Tod1} \\
      C_n(b_{n-1}-b_{n}) &=& (a_n^2)'  \label{Tod2} \\
      a_n^2 C_{n-1} &=& a_{n-1}^2 C_n.  \label{Tod3}
\end{eqnarray} 
From \eqref{Tod3} we find that $a_n^2/C_n$ does not depend on $n$, so that $a_n^2/C_n = a_1^2/C_1$
and from \eqref{Tod1} we find that $C_1=-b_0'(t)$. A simple exercise shows that $b_0'(t)= a_1^2(t)$ so that
$C_n(t) = -a_n^2(t)$. If we use this in \eqref{Tod2}, then we find \eqref{Toda-a}. If we use it in \eqref{Tod1}, then we find \eqref{Toda-b}. 
\end{proof}

The system \eqref{Toda-a}--\eqref{Toda-b} is closely related to a chain of interacting particles with exponential interaction with their
neighbors, introduced by Toda \cite{Toda} in 1967. If $x_n(t)$ is the position of particle $n$, then the Toda system of equations is
\[   x_n''(t) = \exp(x_{n-1}-x_n) - \exp(x_n-x_{n+1}).  \]
The relation with orthogonal polynomials was made by Flaschka \cite{FlashI,FlashII} and Manakov \cite{Mana}, who suggested the change of variables
\[    a_n(t) = \exp(-[x_n-x_{n-1}]/2), \quad   b_n = -x_n'(t) , \]
which gives the system \eqref{Toda-a}--\eqref{Toda-b}. 

If we are dealing with symmetric orthogonal polynomials, i.e., when the measure is symmetric and all the odd moments are zero,
then the three-term recurrence relation simplifies to
\begin{equation}  \label{SM3TRR}
   xP_n(x) = P_{n+1}(x) + a_n^2 P_{n-1}(x), \qquad n \geq 0.
\end{equation}
A symmetric modification of the measure is given by $d\mu_t(x) = e^{tx^2} \, d\mu(x)$ and the relation becomes
\begin{equation}   \label{LangmuirAnsatz}
  \frac{d}{dt} P_n(x;t) = C_n(t) P_{n-2}(x;t). 
\end{equation} 
The compatibility between \eqref{SM3TRR} and \eqref{LangmuirAnsatz} then gives:

\begin{theorem}[Langmuir lattice]
Let $\mu$ be a symmetric positive measure on $\mathbb{R}$ for which all the moments exist and let $\mu_t$ be the measure for which
$d\mu_t(x) = e^{tx^2} \, d\mu(x)$, where $t \in \mathbb{R}$ is such that all the moments of $\mu_t$ exist.
Then the recurrence coefficients of the orthogonal polynomials for $\mu_t$ satisfy the differential-difference equations
\begin{equation}  \label{Lang}
  \frac{d}{dt} a_n^2 = a_n^2 ( a_{n+1}^2 - a_{n-1}^2 ), \qquad   n \geq 1.  
\end{equation} 
\end{theorem}

\begin{proof}
If we differentiate \eqref{SM3TRR} with respect to $t$ and then use \eqref{LangmuirAnsatz}, then we find
\[  xC_n P_{n-2}(x;t) = C_{n+1} P_{n-1}(x;t) + (a_n^2)' P_{n-1}(x;t) + a_n^2 C_{n-1} P_{n-3}(x;t).  \]
Comparing with \eqref{SM3TRR} (with $n$ replaced by $n-2$) gives
\begin{eqnarray}
     (a_n^2)' &=& C_n-C_{n+1},    \label{Lang1} \\
     a_n^2C_{n-1} &=& a_{n-2}^2 C_n.   \label{Lang2}
\end{eqnarray}
From \eqref{Lang2} it follows that $a_n^2a_{n-1}^2/C_n$ is constant and therefore equal to $a_2^2a_1^2/C_2$. Now $C_2(t)= -(a_1^2)'$
and one can easily compute $a_1^2, a_2^2$ and $(a_1^2)'$ in terms of the moments $m_0,m_2,m_4$ to find that $a_2^2a_1^2/C_2 = -1$, so
that $a_n^2a_{n-1}^2 = -C_n$. If one uses this in \eqref{Lang1}, then one finds \eqref{Lang}.
\end{proof}
This differential-difference equation is known as the Langmuir lattice or the Kac-van Moerbeke lattice. 
We will now illustrate this with a number of explicit examples.

\subsection{Discrete Painlev\'e I}
Let us consider orthogonal polynomials for the weight function $w(x)=e^{-x^4+tx^2}$ on $(-\infty,\infty)$.
The symmetry $w(-x)=w(x)$ of this weight function implies that the recurrence coefficients $b_n$ in \eqref{3TRR} or \eqref{M3TRR} vanish
and the three-term recurrence relation is \eqref{SM3TRR}. 
The orthogonal polynomials also have a nice differential property: the \textit{structure relation} is
\begin{equation}  \label{SR}
     P_n'(x) = A_nP_{n-1}(x) + C_nP_{n-3}(x),   
\end{equation}
for certain sequences $(A_n)_n$ and $(C_n)_n$.
Indeed, we can express $P_n'$ in terms of the orthogonal polynomials as
\[   P_n'(x) = \sum_{k=0}^{n-1} c_{n,k} P_k(x), \]
where
\[    c_{n,k} \int_{-\infty}^\infty P_k^2(x) e^{-x^4+tx^2}\, dx = \int_{-\infty}^\infty P_n'(x)P_k(x) e^{-x^4+tx^2}\, dx . \]
Using integration by parts gives
\begin{eqnarray*}
    c_{n,k}/\gamma_{k}^2 &=& - \int_{-\infty}^\infty P_n(x) \bigl( P_k(x) e^{-x^4+tx^2} \bigr)' \, dx \\
	    &=& - \int_{-\infty}^\infty P_n(x)P_k'(x) e^{-x^4 + tx^2} \, dx \\
    & & + \int_{-\infty}^\infty P_n(x) P_k(x) (4x^3-2tx) e^{-x^4+tx^2}\, dx, 
\end{eqnarray*} 	
and the last two integrals are zero for $0 \leq k < n-3$ by orthogonality, so that only $c_{n,n-1}$, $c_{n,n-2}$ and $c_{n,n-3}$ are left.
The symmetry of $w$ implies that $P_{2n}(x)$ is an even polynomial and $P_{2n+1}(x)$ is an odd polynomial for every $n$, hence $c_{n,n-2}=0$.
Taking $A_n=c_{n,n-1}$ and $C_n=c_{n,n-3}$ then gives the structure relation.

We now have a recurrence relation \eqref{SM3TRR} which describes the behavior of $P_n(x)$ in the (discrete) variable $n$, and a 
structure relation \eqref{SR} which describes the behavior of $P_n(x)$ in the (continuous) variable $x$.  
Both relations have to be compatible: if we differentiate \eqref{SM3TRR} and then use \eqref{SR} to replace
all the derivatives, then comparing coefficients of the polynomials $p_k$ gives the \textit{compatibility relations} 
\begin{equation}   \label{dPIF}
     4a_n^2 \left(a_{n+1}^2 + a_n^2 + a_{n-1}^2 - \frac{t}{2}\right) = n.   
\end{equation}
This simple non-linear recurrence relation is known as discrete Painlev\'e I ($\textrm{d-P}_{\scriptstyle \textrm{I}}$) and is a special case
of \eqref{d-P1} we gave earlier. 
This particular equation was
already in work of Shohat \cite{Shohat} in 1939, who extended earlier work of Laguerre \cite{Laguerre} from 1885.
Later it was obtained again by Freud \cite{Freud} in 1976, who was unaware of the work of Shohat. 
The special positive solution needed to get the recurrence coefficients was analyzed by Nevai \cite{Nevai} and Lew and Quarles \cite{LQ}.
An asymptotic expansion was found by M\'at\'e-Nevai-Zaslavsky \cite{MNZ}. 
Only later (in 1991) it was recognized as a discrete Painlev\'e equation by Fokas, Its and Kitaev \cite{FIK} who coined the name $\textrm{d-P}_{\scriptstyle \textrm{I}}$. 
Magnus \cite{Magnus} used the extra parameter $t$ and showed that, as a function of $t$, the recurrence coefficient $a_n(t)$ 
satisfies the differential equation Painlev\'e IV, as we will see later.

The discrete Painlev\'e equation \eqref{dPIF} easily allows to find the asymptotic behavior as $n \to \infty$:
\begin{theorem}[Freud]
The recurrence coefficients for the weight $w(x) = e^{-x^4+tx^2}$ on $(-\infty, \infty)$ satisfy
\[         \lim_{n \to \infty} \frac{a_n}{n^{1/4}} = \frac{1}{\sqrt[4]{12}}. \]
\end{theorem} 

Observe that \eqref{dPIF} is a second order recurrence relation, so one needs two initial conditions $a_0$ and $a_1$ to generate
all the recurrence coefficients. It turns out that the recurrence coefficients are a special solution with $a_0=0$ for which all
$a_n$ are positive for $n \geq 1$. This means that there is only one special initial value $a_1$ that gives a positive solution.
Put $x_n = a_n^2$, then (for $t=0$)
\begin{equation}  \label{dPIx}
     x_n(x_{n+1}+x_n + x_{n-1}) = an, \qquad  a=1/4.  
\end{equation}

\begin{theorem}[Lew and Quarles, Nevai]
There is a unique solution of  \eqref{dPIx} for which $x_0=0$ and $x_n >0$ for all $n \geq 1$.
\end{theorem}
Hence one should not use this recurrence relation \eqref{dPIx} to generate the recurrence coefficients starting from
$x_0=0$ and $x_1$, because a small error in $x_1$ will produce a sequence for which not all the terms are positive. A small
perturbation in the initial condition $x_1$ has a very important effect on the solution as $n \to \infty$. This is not unusual
for non-linear recurrence relations. Instead it is better to generate the positive solution by using a fixed point algorithm,
because the positive solution turns out to be the fixed point of a contraction in an appropriate normed space of infinite sequences.
See, e.g., \cite[\S 2.3]{WVA3}.

\subsection{Langmuir lattice and Painlev\'e IV}
We will modify the measure $\mu$ by multiplying it with the symmetric function $e^{tx^2}$, where $t$ is a real parameter. This gives
the Langmuir lattice \eqref{Lang}. 
We can combine this with the discrete Painlev\'e equation \eqref{dPIF} to find a differential equation for $a_n^2(t)$ as
a function of the variable $t$.
Put $a_n^2 = x_n$, then
\begin{eqnarray}
     n &=& 4x_n(x_{n+1}+x_n+x_{n-1}-t/2) ,  \label{een} \\
     {x}_n' &=&   x_n(x_{n+1}-x_{n-1}) ,  \label{twee}
\end{eqnarray}
where the $'$ denotes the derivative with respect to $t$.
Differentiate \eqref{twee} to find
\[    x_n'' = x_n'(x_{n+1}-x_{n-1}) + x_n(x_{n+1}'-x_{n-1}').  \]
Replace $x_{n+1}'$ and $x_{n-1}'$ by \eqref{twee}, then
\[   x_n'' = x_n'(x_{n+1}-x_{n-1}) + x_n \Bigl( x_{n+1}(x_{n+2}-x_n) - x_{n-1}(x_n - x_{n-2}) \Bigr). \]
Eliminate $x_{n+1}$ and $x_{n-1}$ using \eqref{een}--\eqref{twee} to find
\[  {x}_n'' = \frac{({x}_n')^2}{2x_n} + \frac{3x_n^3}{2} - t x_n^2 + x_n \left( \frac{n}2 + \frac{t^2}{8} \right)
    - \frac{n^2}{32x_n}.  \]
This is Painlev\'e IV if we use the transformation $2x_n(t) = y(-t/2)$. This means that Painlev\'e IV has a solution
which can be described completely in terms of the moments of $w(x)=e^{-x^4+tx^2}$, since $a_n^2 = \gamma_{n-1}^2/\gamma_n^2$
and by \eqref{gammaD} $\gamma_n^2 = D_n/D_{n+1}$, where $D_n$ is the Hankel determinant \eqref{Hankeldet} containing the moments.
Notice that all the odd moments $m_{2n+1}$ are zero, and for the even moments one has
\[   m_{2n} = \int_{\mathbb{R}} x^{2n} e^{-x^4+tx^2}\, dx = \frac{d^n}{dt^n} m_0. \]
Hence the special solution $a_n^2(t)$ of Painlev\'e IV
is in terms of $m_0(t)$ only, and this is a special function:
\[     m_0(t) = \int_{-\infty}^\infty e^{-x^4+tx^2}\, dx = 2^{-1/4} \sqrt{\pi} e^{t^2/8} D_{-1/2}(-\sqrt{t/2}), \]
where $D_{-1/2}$ is a parabolic cylinder function.

\subsection{Singularity confinement}
In this section we will explain the notion of singularity confinement for the discrete Painlev\'e I equation
\[     4x_n (x_{n+1}+x_n+x_{n-1}) = n.   \]
From this equation one finds
\[     x_{n+1} = \frac{n}{4x_n} - x_n - x_{n-1}.  \]
If $x_n=0$ then $x_{n+1}$ becomes infinite. This need not be a problem, but problems arise later when we have to add or subtract infinities.
So we need to be careful and suppose that $x_n=\epsilon$ is small. Then
\[   x_{n+1} = \frac{n}{4\epsilon} - \epsilon -x_{n-1}, \]
and
\[  x_{n+2} = -\frac{n}{4\epsilon} +x_{n-1} + \epsilon + \mathcal{O}(\epsilon^2), \]
and 
\[ x_{n+3} = - \epsilon + \mathcal{O}(\epsilon^2), \]
and one more
\[  x_{n+4} = x_{n-1} + \frac{2-8 x_{n-1}^2}{n} \epsilon + \mathcal{O}(\epsilon^2), \]
and for $\epsilon \to 0$ we see that $x_{n+4}$ is finite again and recovers the value $x_{n-1}$ we had before we started to get singularities.
The singularities are confined to $x_{n+1}$ and $x_{n+2}$ and one can continue the recurrence relation from $x_{n+4}$. 
This has some meaning in terms of the orthogonal polynomials for the weight $e^{-x^4}$, but we have to consider this weight on the set
$\mathbb{R} \cup i \mathbb{R}$
and look for orthogonal polynomials $(R_n)_n$ for which
\[   \alpha \int_{-\infty}^\infty R_n(x)R_m(x) e^{-x^4}\, dx + \beta \int_{-i\infty}^{+i\infty} R_n(x)R_m(x) e^{-x^4}\, |dx| = 0, \quad
n \neq m, \]
with $\alpha, \beta >0$. They satisfy the recurrence relation
\[    xR_n(x) = R_{n+1}(x) + c_n R_{n-1}(x)  \]
and the recurrence coefficients $(c_n)_n$ still satisfy \eqref{dPIx} but with initial condition $c_0=0$ and 
$c_1=\frac{(\alpha-\beta) m_2}{(\alpha+\beta)m_0}$. 
If $\alpha=\beta$ then $c_1=0$ generates a singularity for $\textrm{d-P}_{\small \textrm{I}}$ and gives $c_2 = \infty$, hence $R_3$ does not exist
if we define it using \eqref{detP}.
The singularity, however, is confined to a finite number of terms. We have

\begin{property}
For $\alpha=\beta$ one has $D_{4n-1}=D_{4n-2}=0$ for the Hankel determinants, so that $R_{4n-1}$ and $R_{4n-2}$ as defined by
\eqref{detP} do not exist for $n\geq 1$.
Furthermore
\[    R_{4n}(x) = r_n(x^4), \quad  R_{4n+1}(x) = xs_n(x^4).   \]
\end{property}
The polynomials $r_n$ and $s_n$ can be identified as Laguerre polynomials with parameter $\alpha=-3/4$ and $\alpha=1/4$ respectively.
The problem with $R_{4n-1}$ and $R_{4n-2}$ is not so much that they do not exist, but rather that they are not unique.
\medskip

\noindent\shadowbox{\parbox{12cm}{
\begin{exercise}
Show that for every $a\in \mathbb{R}$ the polynomials $(x^2+ax)s_n(x^4)$ are monic polynomials of degree $4n+2$ that are orthogonal to $x^k$ for $0 \leq k \leq 4n+1$, so that the monic orthogonal polynomial $R_{4n+2}$ is not unique.
In a similar way $(x^3+ax^2+bx)s_n(x^4)$ are monic polynomials of degree $4n+3$ that are orthogonal to $x^k$ for $0 \leq k \leq 4n+2$ for every
$a,b\in \mathbb{R}$ so that the monic orthogonal polynomial $R_{4n+3}$ is not unique. 
\end{exercise}}}

\subsection{Generalized Charlier polynomials}
Our next example is a family of discrete orthogonal polynomials $P_n(x)$, which satisfy
\[    \sum_{k=0}^\infty  P_n(k)P_m(k) \frac{c^k}{(\beta)_k k!} = 0, \qquad n \neq m.  \]
Without the factor $(\beta)_k$ the polynomials are the Charlier polynomials, but with the factor $(\beta)_k$ we have a semiclassical family
of discrete orthogonal polynomials. The case $\beta=1$ was investigated in \cite{MamaVA} and the general case in \cite{SmetVA}, 
see also \cite[\S 3.2]{WVA3}. The structure relation for discrete orthogonal polynomials is now in terms of a difference operator instead of
a differential operator. For these generalized Charlier polynomials it is
\begin{equation}   \label{strucGC}
   \Delta P_n(x) = A_n P_{n-1}(x) + B_n P_{n-2}(x), 
\end{equation}
where $\Delta$ is the forward difference operator acting on a function $f$ by 
\[  \Delta f(x) = f(x+1)-f(x), \]
and $(A_n)_n$ and $(B_n)_n$ are certain sequences. If one works out the compatibility of \eqref{M3TRR} and \eqref{strucGC}, then one finds 
\begin{eqnarray*}
    b_n+b_{n-1} -n+\beta  &=& \frac{cn}{a_n^2}, \\
    (a_{n+1}^2-c)(a_n^2-c) &=&  c (b_n-n)(b_n-n+\beta-1).
\end{eqnarray*}  
This corresponds to a limiting case of discrete Painlev\'e with surface/symmetry $D_4^{(1)}$ in Sakai's classification. 

If we put $c=c_0e^t$, then the weights with parameter $c$ are a Toda modification of the weights with parameter $c_0$,   
\[   \frac{c^k}{(\beta)_k k!} = e^{tk} \frac{c_0^k}{(\beta)_k k!} , \]
and hence the recurrence coefficients satisfy the Toda equations given in Theorem \ref{thm:Toda}. 
Put $x_n(t)=a_n^2$ and $y_n(t)=b_n$, then
\begin{eqnarray*}
    (x_n-c)(x_{n+1}-c) &=& c(y_n-n)(y_n-n+\beta-1), \\
    y_n+y_{n-1}-n+\beta &=& \frac{cn}{x_n}, 
\end{eqnarray*}
and if $x_n'=dx_n/dc$, $y_n'=dy_n/dc$, the Toda lattice equations are
\begin{eqnarray*}
            cx_n' &=& x_n(y_n-y_{n-1}), \\
            cy_n' &=& x_{n+1}-x_n.
\end{eqnarray*}
Eliminate $y_{n-1}$ and $x_{n+1}$ (this requires quite a few computations)
%\[  x_n'' = \frac12 \left( \frac{1}{x_n}+ \frac{1}{x_n-c} \right) (x_n')^2 - \frac{cx_n}{x_n-c} + \frac{c^3n^2}{2x_n(x_n-c)}
%  - \frac{c^2n^2}{x_n-c} -2x_n(x_n-c) + \frac{x(cn^2-\beta^2c+2\beta c)}{2(x_n-c)} . \]
and put $x_n = \frac{c}{1-y}$, then $y(c)$ satisfies (after even more computations)  
\[  y'' = \frac12 \left( \frac{1}{2y} + \frac{1}{y-1} \right) (y')^2 - \frac{y'}{c} + \frac{(1-y)^2}{c^2} \left( \frac{n^2y}{2} - \frac{(\beta-1)^2}{2y}
  \right) - \frac{2y}{c} . \]
This is a Painlev\'e V differential equation as in \eqref{P5} with $\delta=0$. Such an equation can always be transformed to Painlev\'e III.  

\subsection{Discrete Painlev\'e II}
We will now give an example of a family of orthogonal polynomials on the unit circle, for which the recurrence coefficients satisfy
a discrete Painlev\'e equation. 
Orthogonal polynomials on the unit circle (OPUC) are defined by the orthogonality relations
\[   \frac{1}{2\pi} \int_{0}^{2\pi} \varphi_n(z) \overline{\varphi_m(z)} v(\theta)\, d\theta = \delta_{m,n}, \qquad
     z=e^{i\theta}, \quad \varphi_n(z)=\kappa_n z^n + \cdots  \]
where $\kappa_n > 0$. We denote the monic polynomials by $\Phi_n = \varphi_n/\kappa_n$. They satisfy a nice recurrence relation
\begin{equation}   \label{RRopuc}
     z\Phi_n(z) = \Phi_{n+1}(z) + \overline{\alpha_n} \Phi_{n}^*(z), 
\end{equation}
where $\Phi_n^*(z) = z^n \overline{\Phi}_n(1/z)$ is the reversed polynomial. The recurrence coefficients
$\alpha_n = -\overline{\Phi_{n+1}(0)}$ are nowadays known as \textit{Verblunsky coefficients}, but earlier they were also known
as Schur parameters or reflection coefficients.
Let $v(\theta) = e^{t\cos \theta}$ for $\theta \in [-\pi,\pi]$. The trigonometric moments for this weight function are modified Bessel functions
\[    \frac{1}{2\pi} \int_0^{2\pi} e^{in\theta} v(\theta)\, d\theta = I_n(t),  \] 
which is why Ismail \cite[Example 8.4.3]{Ismail} calls them \textit{modified Bessel polynomials}.
The symmetry $v(-\theta)=v(\theta)$ implies that $\alpha_n(t)$ are real-valued. If we write
\[   v(\theta) = \hat{v}(z), \quad    z=e^{i\theta}, \]
then 
\[   \hat{v}(z)=\exp \left( t \frac{z+\frac1z}{2} \right),  \]
and this function satisfies the Pearson equation
\[    \hat{v}'(z) = \frac{t}{2} \left( 1 - \frac{1}{z^2} \right) \hat{v}(z) .  \]
As a consequence the orthogonal polynomials satisfy a structure relation:
\begin{property}
The monic orthogonal polynomials for $v(\theta) = e^{t\cos \theta}$ satisfy
\begin{equation}   \label{SRopuc}
       \Phi_n'(z) = n \Phi_{n-1}(z) + B_n \Phi_{n-2}(z), 
\end{equation}
for some sequence $(B_n)_n$. In fact, one has
\[    B_n = \frac{t}{2} \frac{\kappa_{n-2}^2}{\kappa_n^2} .  \]
\end{property}
We now have two equations: the recurrence relation \eqref{RRopuc} and the structure relation \eqref{SRopuc}, and we can check their compatibility.
They will be compatible if the recurrence coefficients satisfy the following non-linear relation:

\begin{theorem}[Periwal and Shevitz \cite{PS}]
The Verblunsky coefficients for the weight $v(\theta) = e^{t\cos \theta}$ satisfy
\[    - \frac{t}{2} (\alpha_{n+1}+\alpha_{n-1}) = \frac{(n+1)\alpha_n}{1-\alpha_n^2}, \]
with initial values
\[   \alpha_{-1} = -1, \quad \alpha_0 = \frac{I_1(t)}{I_0(t)}.  \]
\end{theorem}

Let $x_{n} = \alpha_{n-1}$, then
\begin{equation}   \label{dPII}
     x_{n+1} + x_{n-1} = \frac{\alpha n x_n}{1-x_n^2}  , \qquad \alpha = -\frac{2}{t},
\end{equation}
and this is a particular case of discrete Painlev\'e II ($\textrm{d-P}_{\scriptstyle \textrm{II}}$) given in \eqref{d-P2}.
We need a solution with $x_0=-1$ and $|x_n| < 1$ for $n \geq 1$, because for Verblunsky coefficients one always has $|\alpha_n| < 1$.
Such a solution is unique.

\begin{theorem}
Suppose $\alpha > 0$. Then there is a unique solution of \eqref{dPII} for which
$x_0=-1$ and $-1 < x_n < 1$. 
The solution corresponds to $x_1=I_1(-2/\alpha)/I_0(-2/\alpha)$ and is negative for every $n \geq 0$.
\end{theorem}

A proof of this result can be found in \cite[\S 3.3]{WVA3} for $\alpha >1$; a proof for $0 < \alpha \leq 1$ has not been published
and we invite the reader to come up with such a proof. 
This special solution converges to zero (fast).

%\begin{property}
%The solution of $\textrm{d-P}_{\scriptstyle \textrm{II}}$ with $x_0=1$ and $0 < x_n < 1$ for $n \geq 1$ satisfies
%\[   \frac{1}{\alpha^n n! e^{1/\alpha}} \leq  \frac{1}{\alpha^n n! \sum_{k=0}^n \frac{\alpha^{-k}}{k!}} \leq x_n \leq \frac{4^n n!}{\alpha^n (2n)!} \sim \frac{1}{\sqrt{2}} \left( \frac{e}{\alpha n} \right)^n.  \]
%\end{property}

\subsection{The Ablowitz-Ladik lattice and Painlev\'e III}
The lattice equations corresponding to orthogonal polynomials on the unit circle are the \textit{Ablowitz-Ladik lattice}
equations (or the Schur flow).

\begin{theorem}
Let $\nu$ be a positive measure on the unit circle which is symmetric (the Verblunsky coefficients are real).
Let $\nu_t$ be the modified measure $d\nu_t(\theta) = e^{t\cos \theta}\, d\nu(\theta)$, with $t \in \mathbb{R}$.
The Verblunsky coefficients $(\alpha_n(t))_n$ for the measure $\nu_t$ then satisfy
\[      2{\alpha}_n' = (1-\alpha_n^2)(\alpha_{n+1}-\alpha_{n-1}), \qquad n \geq 0.   \]
\end{theorem}

We can now combine the discrete Painlev\'e II equation 
\[      \alpha_{n+1} + \alpha_{n-1} = \frac{-2n\alpha_n}{t(1-\alpha_n^2)}  \]
with the Ablowitz-Ladik equation
\[       \alpha_{n+1} - \alpha_{n-1} = \frac{2{\alpha}'_n}{1-\alpha_n^2} .  \]
Eliminate $\alpha_{n+1}$ and $\alpha_{n-1}$ to find
\[   {\alpha}''_n = - \frac{\alpha_n}{1-\alpha_n^2} ({\alpha}'_n)^2 - \frac{{\alpha}'_n}{t}
    - \alpha_n(1-\alpha_n^2) + \frac{(n+1)^2}{t^2} \frac{\alpha_n}{1-\alpha_n^2} .  \]

\noindent\shadowbox{\parbox{12cm}{
\begin{exercise}
If one puts $\alpha_n = \frac{1+y}{1-y}$, then show that $y$ satisfies the Painlev\'e V differential equation \eqref{P5} with $\gamma=0$.
\end{exercise}}}

Painlev\'e V with $\gamma=0$ can always be transformed to Painlev\'e III. A direct approach was given by Hisakado \cite{His} and Tracy and Widom \cite{TW}.
They showed that the ratio $w_n(t) = \alpha_n(t)/\alpha_{n-1}(t)$ satisfies Painlev\'e III.

\subsection{Some more examples}
Several more examples have been worked out in the literature the past few years. Here is a short sample.

\subsubsection{Generalized Meixner polynomials}
 These are discrete orthogonal polynomials
\[    \sum_{k=0}^\infty  P_n(k)P_m(k) \frac{(\gamma)_k a^k}{(\beta)_k k!} = 0, \qquad n \neq m, \]
which were considered in \cite{SmetVA,FilVA,Clarkson2}.
Put $a_n^2 = na -(\gamma-1)u_n$, and $b_n = n+\gamma-\beta+a - \frac{\gamma-1}{a} v_n$, then
\begin{eqnarray*}
      (u_n+v_n)(u_{n+1}+v_n) &=& \frac{\gamma-1}{a^2} v_n(v_n-a) \left( v_n - a \frac{\gamma-\beta}{\gamma-1} \right) , \\
      (u_n+v_n)(u_n+v_{n-1}) &=& \frac{u_n}{u_n-\frac{an}{\gamma-1}} (u_n+a) \left( u_n+a \frac{\gamma-\beta}{\gamma-1} \right). 
\end{eqnarray*}
The initial values are
\[   a_0^2=0, \quad b_0 = \frac{\gamma a}{\beta} \frac{M(\gamma+1,\beta+1,a)}{M(\gamma,\beta,a)},  \]
where $M(a,b,z)$ is Kummer's confluent hypergeometric function.       
This is asymmetric discrete Painlev\'e IV or $\textrm{d-P}(E_6^{(1)}/A_2^{(1)})$.
If we put
\[  v_n(a) = \frac{a\Bigl(ay'-(1+\beta-2\gamma)y^2+(n+1-a+\beta-2\gamma)y-n\Bigr)}{2(\gamma-1)(y-1)y} , \]
then
\[   y''=\left( \frac{1}{2y} +  \frac{1}{y-1} \right) (y')^2 - \frac{y'}{a} + \frac{(y-1)^2}{a^2} \bigl( Ay + \frac{B}{y} \bigr) + \frac{Cy}{a}
      + \frac{Dy(y+1)}{y-1}  \]
with
\[   A = \frac{(\beta-1)^2}{2}, \quad B = - \frac{n^2}{2}, \quad C=n-\beta+2\gamma, \quad D = -\frac12 , \]
which is Painlev\'e V given in \eqref{P5}.

 \subsubsection{Modified Laguerre polynomials}  Chen and Its \cite{ChenIts} (see also \cite[\S 4.4]{WVA3}) looked at orthogonal polynomials for the weight function
  $w(x) = x^\alpha e^{-x} e^{-t/x}$ on $[0,\infty)$. This is a modification of the Laguerre weight with an exponential function that has an essential singularity at $0$.
  Put $b_n= 2n+\alpha+1 + c_n$, $a_n^2 = n(n+\alpha) + y_n + \sum_{j=0}^{n-1} c_j$, and $c_n =1/x_n$, then 
   \begin{eqnarray*}
           x_n + x_{n-1} &=& \frac{nt -(2n+\alpha)y_n}{y_n(y_n-t)}, \\
           y_n + y_{n+1} &=& t- \frac{2n+\alpha+1}{x_n} - \frac{1}{x_n^2}.
   \end{eqnarray*}
   This corresponds to the discrete Painlev\'e equation $\textrm{d-P}((2A_1)^{(1)}/D_6^{(1)})$.
   The exponential modification is not of Toda type but belongs to a similar class of modifications (the Toda hierarchy). With some effort one can find
   the differential equation  
 \[  c_n'' = \frac{(c_n')^2}{c_n} - \frac{c_n'}{t} + (2n+\alpha+1) \frac{c_n^2}{t^2} + \frac{c_n^3}{t^2} + \frac{\alpha}{t} - \frac{1}{c_n} \]
which is Painlev\'e III given in \eqref{P3}. 

\subsubsection{Modified Jacobi polynomials} Basor, Chen and Ehrhardt \cite{BCE} (see also \cite[\S 5.2]{WVA3}) considered the weight $w(x) = (1-x)^\alpha(1+x)^\beta e^{-tx}$.
  This is a Toda modification of the weight function for Jacobi polynomials. In this case one has   
  \begin{eqnarray*}
    tb_n &=& 2n+1+\alpha + \beta -t-2R_n, \\ 
   t(t+R_n)a_n^2 &=& n(n+\beta)-(2n+\alpha+\beta)r_n - \frac{tr_n(r_n+\alpha)}{R_n}, 
\end{eqnarray*}
where $r_n$ and $R_n$ satisfy the recurrence relations
\begin{align*}
    2t(r_{n+1}+r_n) &= 4R_n^2 -2R_n(2n+1+\alpha+\beta-t)-2\alpha t, \\
    n(n+\beta)-(2n+\alpha+\beta)r_n &= r_n(r_n+\alpha) \left( \frac{t^2}{R_nR_{n-1}} + \frac{t}{R_n} + \frac{t}{R_{n-1}} \right), 
\end{align*}
and for $y = 1+t/R_n$ one has the differential equation
\begin{multline*}
  y'' = \frac{3y-1}{2y(y-1)} (y')^2 - \frac{y'}{t} + 2(2n+1+\alpha+\beta) \frac{y}{t} - \frac{2y(y+1)}{y-1} \\
   +\ \frac{(y-1)^2}{t^2} \left( \frac{\alpha^2y}{2} - \frac{\beta^2}{2y} \right),  
\end{multline*}
which is Painlev\'e V given in \eqref{P5}.

\subsubsection{$q$-orthogonal polynomials}
  There are also examples of families of $q$-orthogonal polynomials for which one can find $q$-discrete Painlev\'e equations for the recurrence coefficients.
  In this case the structure relation uses the $q$-difference operator $D_q$ for which
  \[   D_qf(x) = \frac{f(x)-f(qx)}{x(1-q)}.  \]
  If we consider the weight  
\[  w(x) = \frac{x^\alpha}{(-x^2;q^2)_\infty (-q^2/x^2;q^2)_\infty}, \qquad x \in [0,\infty) \]
then the recurrence coefficients (after some transformation) satisfy
$q$-discrete Painlev\'e III
\[    x_{n-1}x_{n+1} = \frac{(x_n+q^{-\alpha})^2}{(q^{n+\alpha}x_n+1)^2} . \]
For the weight
\[   w(x) = \frac{x^\alpha (-p/x^2;q^2)_\infty}{(-x^2;q^2)_\infty (-q^2/x^2;q^2)_\infty}, \qquad x \in [0,\infty) \]
one finds $q$-discrete Painlev\'e V
\[   (z_nz_{n-1}-1)(z_nz_{n+1}-1) = \frac{(z_n+ \sqrt{q^{2-\alpha}/p})^2(z_n \sqrt{pq^{\alpha-2}})^2}{(q^{n+\alpha/2-1}\sqrt{p} z_n+1)^2} . \]
and for 
\[   w(x) = x^\alpha(q^2x^2;q^2)_\infty , \qquad x \in \{ q^k, k=0,1,2,3,\ldots \} \]
one again finds $q$-discrete Painlev\'e V. Observe that sometimes the weights are on $[0,\infty)$ but they can also be on the discrete set $\{q^n, n \in \mathbb{N}\}$.
See \cite[\S 5.4]{WVA3} for more details.

\subsection{Wronskians and special function solutions}
There is a good explanation why these Toda modifications of orthogonal polynomials often give rise to Painlev\'e differential equations.
In fact the solutions that we need for the recurrence coefficients are special solutions of the Painlev\'e equations in terms of special functions,
such as the Airy functions, the Bessel functions, parabolic cylinder functions, the confluent hypergeometric function and the hypergeometric function.
Such special function solutions are often in terms of Wronskians of one of these special functions. We can easily explain where these Wronskians
are coming from, by using the theory of orthogonal polynomials. Indeed, we return to our Hankel determinants $D_n$ given in \eqref{Hankeldet}.
They contain the moments $m_n$, which for a Toda modification are
\[    m_n(t) = \int_{\mathbb{R}} x^n e^{xt} \, d\mu(x) = \frac{d^n}{dt^n} \int_{\mathbb{R}} e^{xt} \, \mu(x) = \frac{d^n}{dx^n} m_0(t).  \]
Hence all the moments are obtained from the moment $m_0(t)$ by differentiation, and the Hankel determinant \eqref{Hankeldet} becomes
\[ D_n = \det \begin{pmatrix}  m_0 & m_0' & m_0'' & \cdots & m_0^{(n-1)} \\
                               m_0' & m_0'' & m_0''' & \cdots & m_0^{(n)} \\
                               m_0'' & m_0''' & m_0^{(4)} & \cdots & m_0^{(n+1)} \\
                                  \vdots & \vdots & \vdots & \cdots & \vdots \\
                               m_0^{(n-1)} & m_0^{(n)} & m_0^{(n+1)} & \cdots & m_0^{(2n-2)}
                      \end{pmatrix} \]
which is the Wronskian of the functions $m_0,m_0',m_0'',\ldots,m_0^{(n-1)}$,
\[    D_n = \textup{Wr}(m_0,m_0',m_0'',\ldots,m_0^{(n-1)}).  \]
The recurrence coefficient $a_n^2$ can be expressed in terms of these Hankel determinants as
\[   a_n^2(t) = \frac{\gamma_{n-1}^2}{\gamma_n^2} = \frac{D_{n+1}(t)D_{n-1}(t)}{D_n^2(t)},  \]
where we used \eqref{gammaD}. The recurrence coefficients $b_n$ can also be found in terms of determinants. If we write $P_n(x)=x^n + \delta_n x^{n-1} + \cdots$ and
compare the coefficients of $x^n$ in the recurrence relation \eqref{M3TRR}, then $b_n = \delta_n-\delta_{n+1}$. The coefficient $\delta_n$ can be obtained from \eqref{detP}
from which we see that $\delta_n = -D_n^*/D_n$, where $D_n^*$ is obtained from $D_n$ by replacing the last column $(m_{n-1},m_n,\ldots,m_{(2n-2)})^T$ by moments of one order higher 
$(m_n,m_{n+1},\ldots,m_{2n-1})^T$. 
If we take a derivative of the Wronskian, then
\[   \frac{d}{dt} D_n = \textup{Wr}( m_0,m_0',m_0'',\ldots,m_0^{(n-2)}, m_0^{(n)}) = D_n^*, \]
so that
\[   b_n(t) = \frac{D_{n+1}'(t)}{D_{n+1}(t)} - \frac{D_n'(t)}{D_n(t)}. \]
This gives explicit expressions of the recurrence coefficients $a_n^2(t)$ and $b_n(t)$ in terms of Wronskians generated from one seed function $m_0(t)$.

\subsection*{Acknowledgement}
Many thanks to Mama Foupouagnigni and Wolfram Koepf for organizing the workshop \textit{Introduction to Orthogonal Polynomials and Applications} in Douala, Cameroon, and for
encouraging me to write this survey. Also thanks to Arno Kuij\-laars with whom I am sharing a course on 
\textit{Orthogonal Polynomials and Random Matrices} at KU Leuven, which
was very useful for the material in Section~\ref{sec:RM}.

\end{document}